%% file: yoshida2019_6.tex
\documentclass[a4j,11pt]{article}

\usepackage{amsmath}
\usepackage{amssymb}
\usepackage{amsfonts}
\usepackage{amsthm}
\usepackage{ascmac}
\usepackage{amssymb}
\usepackage{latexsym}
\usepackage{bm}
\usepackage{setspace}

\newcommand{\onehalf}{
\mbox{\Large $\frac{1}{\, 2 \,}$}
}

\newcommand{\sfrac}[2]{
\mbox{\small $\dfrac{#1}{#2}$}
}

\newcommand{\sdfrac}[2]{
\mbox{\scriptsize $\dfrac{#1}{#2}$}
}

\newcommand{\lbpi}{
\mbox{\normalsize $\pi$}
}

\newcommand{\lsblk}[1]{
\mbox{\fontsize{7pt}{0pt}$#1$}
}

\newcommand{\hilb}[1]{
\big({\mathcal H} #1 \big)
}

\newcommand{\fbepr}{
\mathrm{f}{\bm \beta}^{\prime}(a,b)
}

\newcommand{\nusV}{
\mbox{$\nu_{\mbox{\tiny $V$}}$}
}

\newtheorem{theorem}{Theorem}[section]
\newtheorem{proposition}[theorem]{Proposition}
\newtheorem{definition}[theorem]{Definition}

\newtheorem{remark}[theorem]{\it Remark}
\newtheorem{example}[theorem]{\it Example}

\input cardsfig.tex

\begin{document}

\begin{center}
{\Large {\bf Remarks on a free analogue of the beta prime distribution}}\\

\bigskip 

{Hiroaki Yoshida}

\medskip 

Department of Information Sciences, \\
Ochanomizu University,\\
Otsuka, Bunkyo, Tokyo 112-8610 Japan.\\
{\ttfamily yoshida@ocha.ac.jp}

\smallskip

\bigskip 

\begin{minipage}{13cm}
\begin{spacing}{0.8}
{\small 
{\bf Abstract}.
We introduce the free analogue of the classical beta prime 
distribution by the multiplicative free convolution of 
the free Poisson and the reciprocal of free Poisson distributions,
and related free analogues of the classical $F$, $T$, and 
beta distributions. 
We show the rationales of our free analogues via the score functions
and the potentials. We calculate the moments of the free beta 
prime distribution explicitly in combinatorial by using non-crossing 
linked partitions, and demonstrate that the free beta prime distribution 
belongs to the class of the free negative binomials in the free 
Meixner family.\\
}

{\small
}

\end{spacing}

\end{minipage}
\end{center}


\section{Introduction}
\label{intro}

A non-commutative or quantum probability space is a unital algebra 
$\mathcal{A}$  (possibly non-commutative) together with a linear 
functional, 
$\varphi : \mathcal{A} \to \mathbb{C}$ such that 
$\varphi({\bm 1}) = 1$. If $\mathcal{A}$ is a $C^*$-algebra and
$\varphi$ is a state then we call a non-commutative probability space
$(\mathcal{A}, \varphi)$ a $C^*$-probability space. 

An element in $\mathcal{A}$ is regarded as a non-commutative 
random variable. 
The distribution of $X \in \mathcal{A}$ under $\varphi$ is 
the linear functional 
$\mu_X : \mathbb{C}[\mathbb{X}] \to \mathbb{C}$ where 
$\mathbb{C}[\mathbb{X}]$ is the algebra of polynomials in the variable 
$\mathbb{X}$ and $\mu_X$ is determined by 
$\mu_X \big( p(\mathbb{X}) \big) = \varphi \big( p(X) \big)$ for 
all polynomials $p(\mathbb{X}) \in \mathbb{C}[\mathbb{X}]$.
One should note that the distribution of a random variable 
$X \in \mathcal{A}$ is nothing more than a way of describing the 
moments $\varphi(X^k)$.

If $(\mathcal{A}, \varphi)$ is a $C^*$-probability space and 
$X \in \mathcal{A}$ is self-adjoint then the distribution of $X$ 
is given by the compactly supported probability measure $\nu$ 
on $\sigma (X)$, the spectrum of $X$, by 
$$
   \mu_X \big( p(\mathbb{X}) \big) = \int_{\sigma (X)} p(t) d \nu(t) \;
\mbox{ for } p(\mathbb{X}) \in \mathbb{C}[ \mathbb{X} ].
$$
In this case, we will call the measure $\nu$ the distribution of $X$ 
and we will write $\mu_X$ instead of $\nu$.  This is the case we treat, 
that is, the probability measures in this paper are compactly 
supported on ${\mathbb R}$.

\smallskip

The notion of independence can be understood as a rule of calculating 
mixed moments. In classical (usual) probability such a 
rule leads to only one meaningful notion of (commutative) independence.
In non-commutative case, however, one can consider several notions of 
independence. Among them, the free independence, first introduced by 
Voiculescu in \cite{Vo85}, seems to be the most interesting and 
important notion in non-commutative probability.

Presently, the non-commutative probability equipped with free independence 
is called free probability, and its theoretical framework is very similar 
to that of classical probability.
With a notion of independence, one can consider the corresponding 
convolution of probability measures.
Namely, for freely independent random variables $X$ and $Y$ with 
respective distributions $\mu_X$ and $\mu_Y$, we introduce 
the additive free convolution of $\mu_X$ and $\mu_Y$ as 
the distribution of the sum of the random variables $X + Y$.

In classical probability the product of independent random 
variables is rather trivial. But in free probability it has a much 
richer structure because free independence is highly non-commutative.
Hence it is worth defining the multiplicative free convolution as 
the distribution of the product of freely independent random 
variables. 
The precise definitions of free independence and the corresponding 
convolutions are mentioned in Section 2 below. The same section 
introduces some analytic tools for the calculation of the additive 
free and the multiplicative free convolutions.

\smallskip

It is a natural approach in free probability to look for the free 
analogue of the classical distributions with the same probabilistic 
framework as in classical probability. For example, Wigner's 
semicircle law plays the Gaussian law role in free probability, 
which can be obtained as the free central limit distribution, and 
the Marchenko-Pastur distribution appears as the limit distribution 
of the free Poisson limit, hence it is often called the free 
Poisson distribution.

\smallskip

Surprisingly many classical characterization problems of the 
probability measures by independence have free probability counterparts. 
For example, 
the free analogue of Bernstein's theorem was proved in \cite{Ni96}, 
which says that for freely independent random variables $X$ and $Y$, 
the random variables $X + Y$ and $X - Y$ are freely independent 
if and only if $X$ and $Y$ are distributed according to the 
semicircle law. 
Moreover, the characterization of the semicircle law by free 
independence of linear and quadratic statistics in freely i.i.d. 
random variables was investigated in \cite{EL17} 
(see also \cite{HKNY99}).

\smallskip

Another free analogue of characterization problem by independence, 
in \cite{Sz15}, \cite{Sz16}, that is, a free version of Lukacs theorem, 
the classical version states that, for 
independent random variables $X$ and $Y$, the random variables $X + Y$ 
and $\sfrac{X}{X + Y}$ are independent if and only if $X$ and $Y$ 
are gamma distributed with the same scale parameter \cite{Lu55}.
Here one should note that in the free analogue of Lukacs theorem, 
the free Poisson distribution plays the classical gamma distribution
role.  However, it is not very strange because the 
free Poisson distribution can be thought as the free 
$\chi^2$-distribution and the classical $\chi^2$-distribution is 
in a family of the classical gamma distributions.
This phenomenon can be found in yet another characterization 
problem by free independence, namely, in the study of the free 
analogue of Matsumoto-Yor property in \cite{Sz17}, which 
elucidated that the role of the classical gamma distribution is 
taken again by the free Poisson distribution.

\smallskip

Furthermore, in the free probability literature there are 
investigations on the free analogue of the characterization 
problems of the probability measure by regression conditions.
A well-known example of such a direction is the characterization 
of the free Meixner distributions defined in \cite{An03}
(see also \cite{SY01}). 
Particularly, in the papers \cite{BoBr06} and \cite{Ej12}, the authors 
studied the free analogue of Laha-Lukacs regression problem, and 
found that the free Meixner distributions can be classified in 
six classes: semicircle, free Poisson, free Pascal
(free negative binomial), free gamma, free binomial, and 
pure free Meixner distributions, which corresponds to the result 
in the classical case. 
Inspired by this result, free analogue of characterization problems 
by regression were further investigated in \cite{Ej14}, \cite{EFS17}, 
and \cite{SW14}. 

\smallskip

In this paper we look for the free analogue of the classical 
distribution as in the literature of free probability to date.

In particular, the free beta prime distribution is investigated. 
In classical probability a certain dilation of the beta prime 
distribution gives the $F$-distribution, thus the 
free $F$-distribution can be also obtained immediately.
Broadly speaking we define the free beta prime distribution as 
the distribution of the ratio of freely independent pair of free 
Poisson random variables that are based on the property that the 
distribution of the ratio of independent classical gamma distributed 
random variables yields the beta prime distribution.
In other words, we will make the free Poisson random variables 
play the free gamma random variables role again. 
The free $T$-distribution and the free beta distribution can be 
also derived from the specialized free $F$-distribution and the free
beta prime distribution, respectively, via some transformations.

Furthermore, by using the combinatorial object, the non-crossing 
linked partitions, first introduced in \cite{Dy07} 
with some set partition statistics, we will investigate the moments 
of the free beta prime distributions, and show that the free beta 
prime distributions can be classified into the free negative 
binomials of the free Meixner family.

\smallskip

The paper is organized as follows, in Section 2 we recall the 
definitions of free independence and of the corresponding additive 
and multiplicative convolutions, and we provide the analytic tools 
for calculation of these convolutions. 
We introduce the free beta prime distribution and some related 
analogues in Section 3; the rationales of which are discussed 
by the score functions in Section 4. 
In Section 5, we review the non-crossing linked partitions and 
some set partition statistics, and we give the combinatorial moment 
formula related to the multiplicative free convolution by using 
these set partition statistics. 
We investigate the moments of the free beta prime distributions 
in Section 6 and determine the class of the free beta prime 
distributions.

\medskip

%
%
\section{Preliminaries on free probability}
\label{sec:2}

In this section, we recall the definition of free independence 
and analytic tools for the calculation of the corresponding additive
and multiplicative convolutions. The introduction here is far from 
being detailed. A comprehensive and good introduction to the theory 
of free probability can be found in the monographs \cite{MS17},
\cite{NS06}, and \cite{VDN92}.

\subsection{Free independence and additive free convolution}
\label{sec:2_1}

A family of subalgebras $\big( \mathcal{A}_i \big)_{i \in I}$, where 
${\bm 1} \in \mathcal{A}_i$, 
in the non-commutative probability space $\big( \mathcal{A}, \varphi \big)$ 
is freely independent if $\varphi(X_1 \, X_2 \cdots X_n) = 0$ whenever 
$\varphi(X_k) = 0$, $1 \le k \le n$ and $X_k \in \mathcal{A}_{j(k)}$ 
where $j(1) \ne j(2) \ne \cdots \ne j(n)$, that is, consecutive 
indices are distinct.

We say that a family of random variables $\big( X_i \big)_{i \in I}$ 
are freely independent (simply, free) if the subalgebras generated by 
$\{ X_i, {\bm 1} \}$ are freely independent.

\smallskip

Let $X$ and $Y$ be free independent random variables in 
$(\mathcal{A}, \varphi)$ with respective distributions 
$\mu_X$ and $\mu_Y$. Then the distribution of $X + Y$ is 
completely determined by the distributions $\mu_X$ and $\mu_Y$.
Hence we define the additive free convolution operation $\boxplus$ 
of $\mu_X$ and $\mu_Y$ by 
$$
   \mu_X \boxplus \mu_Y = \mu_{X+Y}.
$$
Here we should note that the additive free convolution of distributions 
depends only on the distributions, not on the choice of particular 
random variables having those distributions, thus the operation 
$\boxplus$ is well-defined. 

\smallskip

In order to calculate the additive free convolution, that is,
to find the higher moments $\varphi \big( (X + Y)^n \big)$, 
Voiculescu invented the $R$-transform in \cite{Vo86}.

For compactly supported probability measure $\mu$ on ${\mathbb R}$, 
the $R$-transform of $\mu$ is the power series 
$$
   R_\mu (z) = \sum_{n = 1}^{\infty} r_{n} z^{n-1} 
$$
defined as follows.  Consider the Cauchy transform $G_\mu$ of 
the compactly supported probability measure $\mu$ on ${\mathbb R}$
$$
   G_\mu (z) = \int_{\mathbb R} \frac{d \mu(x)}{z - x}, 
$$
which is an analytic function in 
$\mathbb{C} \setminus \mathrm{supp}(\mu)$. 
We take the inverse of $G_\mu (z)$ in a neighborhood of $\infty$ 
because $G_\mu (z)$ is univalent and analytic there.

Then the $R$-transform of $\mu$ can be obtained as 
$$
  R_{\mu} (z) 
    = G_{\mu}^{\mbox{\tiny $\langle  {-1} \rangle $}} (z) 
         - \frac{1}{\, z \,}, 
$$
which is an analytic function in a neighborhood of $0$, 
where $G_{\mu}^{\mbox{\tiny $\langle  {-1} \rangle $}}$ 
denotes the inverse of $G_{\mu}$.

The $R$-transform is the linearizing functor for the additive 
free convolution, that is, 
$$
  R_{\mu \boxplus \nu} (z) = R_{\mu} (z) + R_{\nu} (z).
$$
Hence the coefficients $\{ r_n \}$ of the $R$-transforms are called the 
free cumulant.

\subsection{Full Fock space and canonical random variables}
\label{sec:2_2}

An example of free independence which plays an important role in 
the theory of free probability is that of the creation and the 
annihilation operators on full Fock space.

\smallskip

  Let $\mathcal{H}$ be a Hilbert space and let $\mathcal{T}(\mathcal{H})$ 
denote the full Fock space of $\mathcal{H}$ 
$$
  \mathcal{T}(\mathcal{H}) = \mathbb{C} \Omega \oplus 
            \Big( \bigoplus_{n = 1}^{\infty} \mathcal{H}^{\otimes n} \Big),
$$
where 
{\small 
$\displaystyle{\mathcal{H}^{\otimes n} = 
       \underbrace{\mathcal{H} \otimes \cdots 
       \otimes \mathcal{H}}_{\footnotesize n-\mbox{times}}}$
}
and $\Omega$ is the distinguished unit vector called vacuum.
The full Fock space $\mathcal{T}(\mathcal{H})$ is endowed with 
the inner product
$$
  \langle \, 
    x_1 \otimes x_2 \otimes \cdots \otimes x_n \, |  \, 
    y_1 \otimes y_2 \otimes \cdots \otimes y_m \,
  \rangle
  = \delta_{n,m} \prod_{i = 1}^{n} (\, x_i \,| \, y_i \,), 
$$
where $(\, \cdot \,| \,  \cdot \,)$ is
the inner product on the Hilbert space $\mathcal{H}$.

\smallskip

For $h \in \mathcal{H}$, let 
$\ell(h) \in \mathcal{B} \big( \mathcal{T}(\mathcal{H})\big)$ 
be the left creation operator so that 
$\ell(h) \xi = h \otimes \xi$, and the adjoint operator $\ell(h)^{*}$ 
is called the left annihilation operator.
We consider the state $\varphi(T) = \langle T \Omega, \Omega \rangle$ 
on $\mathcal{B} \big( \mathcal{T}(\mathcal{H})\big)$ given by
the vacuum vector $\Omega$, which is called the vacuum expectation.
The following important example of a free family of subalgebras and 
the distinguished operators in 
$\mathcal{B} \big( \mathcal{T}(\mathcal{H})\big)$ can be found in 
\cite{Vo86}.

\noindent
\begin{theorem}\label{th:2_1}
Let $( e_i )_{i \in I}$ be a family of orthonormal vectors 
in $\mathcal{H}$. Then the family of subalgebras 
$\big( \mbox{alg} \{ \ell(e_i), \ell(e_i)^{*} \} \big)_{i\in I}$ 
is free with respect to the vacuum expectation.
\end{theorem}

\begin{proposition}
Let $e_1$ and $e_2$ be orthonormal unit vectors in $\mathcal{H}$, 
and $\ell_1 = \ell(e_1)$ and $\ell_2 = \ell(e_2)$ be the 
left creation operators in 
$\mathcal{B} \big( \mathcal{T}(\mathcal{H})\big)$. 
Let further 
$$
 \begin{aligned}
   X_1 & = \ell_1 + p_1 {\bm 1} + p_2 \ell_2^{*} 
             + p_3 \big( \ell_1^{*} \big)^2 + \cdots, \\
   X_2 & = \ell_2 + q_1 {\bm 1} + q_2 \ell_2^{*} 
             + q_3 \big( \ell_2^{*} \big)^2 + \cdots.
 \end{aligned}
$$
Then the random variables $X_1 + X_2$ and 
$$
   X_3 = \ell_1 + (p_1 + q_1) {\bm 1} 
                + (p_2 + q_2)  \ell_1^{*} 
                + (p_3 + q_3) \big( \ell_1^{*} \big)^2 + \cdots
$$
have the same distribution in 
$\big( \mathcal{B} \big( \mathcal{T}(\mathcal{H})\big), \varphi \big)$.
\end{proposition}\label{prop:2_2}

Since $\ell_1$ and $\ell_2$ are free with respect to the vacuum 
expectation, the above proposition gives the model for the additive 
free convolution. 
Namely, we denote the $R$-transform of a compactly supported probability 
measure $\mu$ on ${\mathbb R}$ by $R_\mu (z)$ then the random variable 
$$
   X = \ell + R_\mu \big(  \ell^{*} \big) 
$$
on $\mathcal{T}(\mathcal{H})$ is called 
the canonical operator for the additive free convolution. 
In this case, of course, the probability distribution of $X$ is 
given by $\mu$.

\subsection{Multiplicative free convolution}
\label{sec:2_3}

As we have mentioned in Introduction that by high non-commutativity 
of freeness, the product of freely independent random variables 
has a richer structure than the classical case. 
Here one can introduce the multiplicative free convolution 
operation on compactly supported provability measures on 
$\mathbb{R}_{\ge 0} = [0, \infty)$. 

\smallskip 

Let $X$ and $Y$ be freely independent random variables in 
$(\mathcal{A}, \varphi)$ with respective distributions 
$\mu_X$ and $\mu_Y$ compactly supported on $\mathbb{R}_{\ge 0}$. 
Then the distribution of 
$X^{\frac{1}{\, 2 \,}} Y X^{\frac{1}{\, 2 \,}}$ is 
determined only by the distributions $\mu_X$ and $\mu_Y$.
Thus we can define the multiplicative free convolution 
operation $\boxtimes$ of $\mu_X$ and $\mu_Y$ by 
$$
   \mu_X \boxtimes \mu_Y 
 = \mu_{X^{\frac{1}{\, 2 \,}} Y X^{\frac{1}{\, 2 \,}}}.
$$
The multiplicative free convolution is well-defined and becomes 
the commutative operation on compactly supported probability 
measures on ${\mathbb R}_{\ge0}$.

\smallskip

An analytic tool for calculation of the multiplicative free 
convolution is the $S$-transform invented by Voiculescu in \cite{Vo87},
 which is multiplicative map and plays the Mellin transform role 
in classical probability.

\smallskip 

Let $\mu$ be a compactly supported probability measure 
on $[0, \infty)$ with non-zero mean $m_1(\mu) \ne 0$.
We shall first introduce the $\Phi$-series of $\mu$, which is 
essentially the moment generating function but without constant
$$
  \Phi_\mu (z) = \sum_{n = 1}^{\infty} m_n (\mu) z^n 
      = \sfrac{1}{\, z \,} G_\mu \Big( \sfrac{1}{\, z \,} \Big) -1,
$$
where $m_n (\mu)$ stands for the $n$th moment of the probability 
measure $\mu$ and we denote by $G_\mu$ the Cauchy transform of $\mu$.

Then the $S$-transform and the $\Phi$-series are related by 
the functional relation 
$$
   \Phi_\mu \Big( \sfrac{z}{z + 1} S_\mu (z) \Big) = z.
$$
Thus the $S$-transform can be obtained as 
$$
  S_{\mu} (z) 
    = \sfrac{z + 1}{z} 
      \Phi_{\mu}^{\mbox{\tiny $\langle {-1} \rangle $}}(z).
$$
The $S$-transform is multiplicative map for the multiplicative 
free convolution, that is 
$$
 S_{\mu \boxtimes \nu} (z) = S_{\mu} (z) S_{\nu} (z).
$$

\smallskip 

Like in the additive case, the canonical operator for the 
multiplicative free convolution is known as follows.
Let $\mu$ be a compactly supported probability measure 
on $[0, \infty)$ with non-zero mean, we introduce 
the $T$-transform of $\mu$ by the reciprocal of 
the $S$-transform of $\mu$, that is, 
$$
    T_\mu (z) = \frac{1}{S_\mu (z)},
$$
the remarkable combinatorial properties of which were investigated 
by Dykema in \cite{Dy07}. Some of detailed results are discussed in 
Section 5 below. 

Then the canonical operators for the multiplicative case can be 
given by the operator on $\mathcal{T}(\mathcal{H})$ of the form 
$$
   X = ({\bm 1} +  \ell) T_\mu \big(  \ell^{*} \big).
$$

\medskip 

%
%
\section{The free beta prime distribution 
and some derived free analogues}
\label{sec:3}

\subsection{The free beta prime distribution}
\label{sec:3_1}

Let $P_\lambda$ be a free Poisson random variable of parameter 
$\lambda >0$, which has the probability distribution $\mu_\lambda$, 
the Marchenko-Pastur law of parameter $\lambda$, 
$$
   d \mu_\lambda(x) = 
   \frac{ \sqrt{ - (x - \lambda_{-})( x - \lambda_{+})}}{2 \pi x} 
   I_{[\lambda_{-}, \lambda_{+}]}(x) \, dx + \max(1-\lambda, 0) \delta_0 (x),
$$
where $\lambda_{\pm} = \big( 1 \pm  \sqrt{\lambda} \big)^2$ and $I$ stands
for the indicator function.

Then it is well-known that the Cauchy transform of $\mu_\lambda$ is 
given by 
$$
   G_{\mu_\lambda} (z) 
   = \frac{z + (1 - \lambda) 
        - \sqrt{\big( z + (1 - \lambda) \big)^2 - 4 z}}{2z},
$$
where the branch of the square root should be chosen so that 
$\displaystyle{\lim_{z \to \infty} G_{\mu_\lambda} (z) = 0}$.

For a self-adjoint random variable $X$ with the probability 
distribution $\mu$ on ${\mathbb R}$, we denote the Cauchy 
transform of $\mu$ by an abbreviated notation $G_X$. 
We shall apply such a notation to other characteristic series 
of the distribution of a non-commutative random variable. 

\smallskip 

Since we will calculate a multiplicative free convolution involving 
$P_\lambda$, we need the $S$-transform of $P_\lambda$, and it is given by 
$$
   S_{P_\lambda} (z) = \frac{1}{z + \lambda},
$$
which can be found in, for instance, \cite{BBCC11}.

\smallskip 

The following Proposition direct consequence of Proposition 3.13 
in \cite{HaSc07}. 
\begin{proposition}
Let $P_\lambda$ be a free Poisson random variable of parameter $\lambda$.
If the parameter $\lambda > 1$, then $P_\lambda$ is strictly positive and 
invertible, and the $S$-transform of the probability distribution of 
$P_\lambda^{-1}$ is given by 
$$
  S_{P_\lambda^{-1}} (z) = - z + (\lambda - 1).
$$
\end{proposition}

\smallskip

Now we can introduce the free analogue of the beta prime distribution 
by the free multiplicative convolution of the distributions of 
free Poisson and the reciprocal of free Poisson random 
variables. 

\smallskip

Let $P_a$ and $P_b$ be freely independent free Poisson random variables 
of parameter $a$ and $b$, respectively.
We assume parameter $b > 1$ while $a > 0$. 
Since $b > 1$ the random variable $P_b$ is strictly positive, 
we can define the self-adjoint operator 
$$
  X(a,b) = P_b^{- \frac{1}{\,2\,}} \, P_a \, P_b^{- \frac{1}{\,2\,}}, 
$$
and its probability distribution is given by the multiplicative free 
convolution of the distributions of $P_a$ and $P_b^{-1}$, 
$$
 \mu_{X(a,b)} = \mu_{P_a} \boxtimes \mu_{P_b^{-1}}.
$$

\begin{proposition}
The Cauchy transform of the distribution of $X(a,b)$ is given by 
$$
 G_{X(a,b)} (z) 
 = \frac{ (b + 1) z + (1 - a) - 
     \sqrt{(b - 1)^2 z^2 - 2 (a b + a + b -1) z + (a -1)^2}}{2 z (1 + z)},
\eqno{(1)}
$$
where the branch of the square root is chosen so that 
$\displaystyle{\lim_{z \to \infty} G_{X(a,b)} (z) = 0}$.

\end{proposition}

\begin{proof} 
By the multiplicativity, the $S$-transform of $X(a,b)$ 
is given by 
$$
   S_{X(a,b)} (z) = S_{P_a} (z) S_{P_b^{\mbox{\tiny$-1\!$}}} (z) 
   = \frac{- z + (b - 1)}{z + a}.
\eqno{(2)}
$$
Thus we can calculate the moment generating function of $X(a,b)$ 
as follow.
$$
 \begin{aligned}
    \Phi_{X(a,b)}^{\mbox{\tiny$\langle -1 \rangle$}} (z)
     & = \frac{z}{1 + z} S_{X(a,b)} (z) 
       = \frac{ z (b -1 - z)}{(1 + z) (z + a)}, \\
    \Phi_{X(a,b)} (z) 
     & = \frac{ (b - 1) - (1 + a)z - 
         \sqrt{((b - 1) - (1 + a)z )^2 - 4 a z (z + 1) }}{2 (1 + z)},  
\; \mbox{ and } &  \\
    M_{X(a,b)} (z) & = \Phi_{X(a,b)} (z) + 1 \\
     & = \frac{ (b + 1) + (1 - a) z - 
         \sqrt{((b - 1) - (1 + a)z )^2 - 4 a z (z + 1) }}{2 (1 + z)}.
 \end{aligned}
$$
Eventually, the Cauchy transform is given by 
$$
 \begin{aligned}
    G_{X(a,b)} (z) & = \sfrac{1}{\, z \,} 
                  M_{X(a,b)} \Big( \sfrac{1}{\, z \,} \Big) \\
    & = \frac{ (b + 1) z + (1 - a) - 
           \sqrt{(b - 1)^2 z^2 - 2 (a b + a + b -1) z + (a -1)^2}}
             {2 z (1 + z)}.
 \end{aligned}
$$
\end{proof}

\medskip 

Applying Stieltjes inversion formula, one can obtain the probability 
measure of the distribution of $X(a,b)$.

\smallskip 

Concerning the point mass of the measure, the function $G_{X(a,b)} (z)$ 
has the simple poles at $z = 0$ and $z = -1$ and the residues 
are calculated as 
$$
   {\mathrm {Res}} \big( G_{B} (z) ; z = 0 \big) = \max (1 - a,0) 
  \; \mbox{ and } \; 
   {\mathrm {Res}} \big( G_{B} (z) ; z = -1 \big) = 0,
$$
respectively. 
Thus the measure has a point mass $(1 - a)$ at $0$ if $0 < a < 1$,
while $z = -1$ is removable singularity.

\smallskip 

The density $f_{\mbox{\scriptsize f}{\bm \beta}^\prime}(x ; a, b)$ 
of the Lebesgue absolutely continuous part is given by 
$$
 \begin{aligned}
  f_{\mbox{\scriptsize f}{\bm \beta}^\prime}(x ; a, b) 
   = &  - \lim_{\varepsilon \to +0} \frac{1}{\, \pi \,} 
   \mathrm{Im} \big( G_{X(a,b)} ( x + \varepsilon \mathrm{i}) \big) \\
   = & \frac{\sqrt{ - (b - 1)^2 x^2 + 2 (a b + a + b -1) x - (a -1)^2 }}
            {2 \pi x (1 + x)} 
 \end{aligned}
$$
which is supported on the interval $[\gamma_{-}, \gamma_{+}]$, 
where 
$\displaystyle{ \gamma_{\pm} 
        = \bigg( \sfrac{\sqrt{a b } \pm 
         \sqrt{a + b -1}}{b - 1} \bigg)^{\!2}
}$ 
are the roots of the quadratic equation
$(b - 1)^2 x^2 - 2 (a b + a + b -1) x + (a -1)^2 =0$.

\smallskip 

\begin{definition}{\rm
 For $a > 0$ and $b > 1$, the free beta prime distribution 
${\mathrm{f}{\bm \beta}}^{\prime} (a, b)$ 
of parameters $a$ and $b$ is the probability measure 
$$
  d \mu_{a,b} (x) = 
       \frac{(b - 1) \sqrt{- (x - \gamma_{+}) (x - \gamma_{-})} }
       {2 \pi x (1 + x)} \, I_{[\gamma_{-}, \gamma_{+}]}(x) \, dx 
       + \max(1 - a, 0)\, \delta_0 (x),
$$
where $\displaystyle{ \gamma_{\pm} 
        = \bigg( \sfrac{\sqrt{a b } \pm 
         \sqrt{a + b -1}}{b - 1} \bigg)^{\!2}}$.
}
\end{definition}

\subsection{The free $F$-distribution}
\label{sec:3_2}

Let $X_1$ and $X_2$ be the classical $\chi^2$-distributed random
variables of degree of freedom $d_1$ and $d_2$, respectively, and 
assume $X_1$ and $X_2$ are independent.
Then it is well-known that the probability distribution of random 
variable of the scaled 
ratio $Y = \sfrac{X_1/d_1}{X_2/d_2} 
   = \Big( \sfrac{d_2}{d_1} \Big) \sfrac{X_1}{X_2}$ is 
the $F$-distribution $F(d_1, d_2)$.

\smallskip 

As mentioned in the next section with respect to the rationales
of free analogue, the free Poisson law of parameter $\lambda$ can be 
regarded as the free $\chi^2$-distribution of degree of 
freedom $\lambda$, and we have introduced the free beta prime 
distribution ${\mathrm{f}{\bm \beta}^{\prime}}(a,b)$ essentially based 
on the ratio $\sfrac{P_a}{P_b}$ of freely independent pair of free 
Poisson random variables.
With these facts in mind, it is natural to introduce the free 
$F$-distribution in the following manner:

\begin{definition}{\rm
For $a > 0$ and $b > 1$, let $X(a,b)$ be a free beta prime 
${\mathrm{f}{\bm \beta}^{\prime}}(a,b)$-distributed random variable.
The free $F$-distribution $\mathrm{f}F(a,b)$ is defined 
as the distribution of the scaled random variable 
$\sfrac{b}{\, a \,} \, X(a,b)$, the probability measure 
of which is given by 
$D_{\frac{b}{\, a \,}} \big( \mu_{a,b} \big)$, where 
$D_{\frac{b}{\,a\,}}$ stands for the dilation.
}
\end{definition}

\smallskip

Here we shall give the measure 
$\nu_{a,b} = D_{\frac{b}{\,a\,}} \big( \mu_{a,b} \big)$ 
explicitly. If $0 < a < 1$ then $\nu_{a,b}$ has the point mass 
at $0$ with mass $(1 - a)$. 
The density for the Lebesgue absolutely continuous part 
of $\nu_{a,b}$ is given by 
$$
 \begin{aligned}
  \sfrac{a}{\, b \,} \, 
  f_{\mbox{\scriptsize f}{\bm \beta}^\prime}
  \Big( \sfrac{a}{\, b \,} x \, ; a, b  \Big)
= & \Big( \sfrac{a}{\, b \,} \Big) 
    \frac{(b - 1) \sqrt{- \big\{ \big( \frac{a}{\, b \,} \big) x - 
                                              \gamma_{+} \big\} 
                         \big\{ \big( \frac{a}{\, b \,} \big) x - 
                                              \gamma_{-} \big\} } }
    {2 \pi \big( \frac{a}{\, b \,} \big) \, x \,  
           \big\{ 1 + \big( \frac{a}{\, b \,} \big) \, x \big\} } \\
= & \frac{(b - 1) \sqrt{- \big\{ x - 
                       \big( \frac{b}{\, a \,} \big) \gamma_{+} \big\} 
                         \big\{ x - 
                       \big( \frac{b}{\, a \,} \big) \gamma_{-} \big\} } }
    {2 \pi \, x \, \big\{\big( \frac{b}{\, a \,} \big) +  x \big\} }.
 \end{aligned}
$$
Hence we have 
$$
 d \nu_{a,b} (x) = \frac{(b - 1) \sqrt{- (x - \eta_{+}) (x - \eta_{-})} }
            {2 \pi x \big\{\big( \frac{b}{\, a \,} \big) +  x \big\}}
            \, I_{[\eta_{-}, \eta_{+}]} dx 
   + \max \big(1 - a, 0 \big)\, \delta_0 (x), 
$$
where 
$$
  \eta_{\pm} 
   = \frac{b}{\, a \,} \, \bigg( \frac{\sqrt{a b } \pm 
                  \sqrt{a + b -1}}{b - 1} \bigg)^2.
$$

\begin{remark}{\rm
The free $F$-distribution introduced above is not very new one  
but it has been known as the limit spectral distribution of a multivariate 
$F$-matrix (random Fisher matrix) which is one of the important models 
of random matrices, see Chapter 4 in \cite{BS10}.  
In our setting, their result can be stated as follows:
Let the random matrix ${\bm F} = {\bm S}_{n_1} {\bm S}_{n_2}^{-1}$, 
where ${\bm S}_{n_i}$ $(i = 1, 2)$ is a sample covariance matrix with 
dimension $p$ and sample size $n_i$ with an underlying distribution of 
mean $0$ and variance $1$.
If ${\bm S}_{n_1}$ and ${\bm S}_{n_2}$ are independent and we take the 
limit with the asymptotic ratios $p/n_1 \to 1/a \in (0, \infty)$ and 
$p/n_2 \to 1/b \in (0, 1)$. 
Then the limit spectral distribution of the random matrix ${\bm F}$ is 
given by the free $F$-distribution $\mathrm{f}F(a,b)$.
}
\end{remark}

\subsection{The free $T$-distribution}
\label{sec:3_3}

For the free $F$-distribution $\mathrm{f}F(a,b)$, we consider the 
critical case that $a = 1$, which corresponds to the free analogue of 
one-dimensional $T^2$-distribution. Because in classical probability,
it is known that if a random variable $Y$ is distributed according to 
the $T$-distribution of degree of freedom $m$, then the random variable 
$X = Y^2$ is $F(1, m)$-distributed.

Hence it is natural to understand that the distribution of a symmetric 
random variable $Y$ such that $Y^2$ has the free $F$-distribution 
$\mathrm{f}F(1,m)$ is the free $T$ -distribution of degree of 
freedom $m$. 
Here the fact that $Y$ is symmetric means that all the odd moments 
$m_{2k - 1} (Y)$ $= E(Y^{2k-1})$, $k \ge 1$ vanish.


\begin{proposition} 
We assume $m > 1$, and let $Y$ be a symmetric random variable 
such that the distribution of 
$Y^2$ is the free $F$-distribution $\mathrm{f}F(1,m)$.
Then $G_Y (z)$, the Cauchy transform of the distribution of $Y$, is 
given by 
$$
   G_Y (z) = \frac{ (m + 1) z - \sqrt{(m - 1)^2 z^2 - 4 m^2  }}
                     {2 (m + z^2) }.
\eqno{(3)}
$$
\end{proposition}

\begin{proof}
Let $M_Y (z)$ and $M_{Y^2}(z)$ be the moment generating 
functions of the random variables $Y$ and $Y^2$, respectively.
Since $Y$ is symmetric, these generating functions satisfy the relation
$$
   M_Y (z) = \sum_{\ell \ge 0} m_{\ell} (Y) z^{\ell}  
           = \sum_{k \ge 0} m_{2 k} (Y) z^{2 k}
           = \sum_{k \ge 0} m_{k} (Y^2) \big( z^2 \big)^k
           = M_{Y^2} (z^2),
$$
which implies the relation between the Cauchy transforms that 
$$
     \sfrac{1}{\, z \,} G_Y \Big( \sfrac{1}{\, z \,} \Big)
   = \sfrac{1}{\, z^2 \,} G_{Y^2} \Big( \sfrac{1}{\, z^2 \,} \Big),
\; \mbox{ equivalently } \; 
     G_Y ( z ) = z \, G_{Y^2} \big( z^2 \big).
$$
By the assumption the Cauchy transform $G_{Y^2} (z)$ can be 
obtained by one for the scaled free beta prime distributed random 
variable $m X(1,m)$
$$
   G_{m X(1,m)} (z) 
 = \sfrac{1}{\, m \,} \, G_{X(1,m)} \Big( \sfrac{z}{\, m \,} \Big),
$$
where $G_{X\!(1,m)}$ is the Cauchy transform of the free 
beta prime distribution $\mathrm{f}{\bm \beta}^{\prime}(1,m)$ 
given by (1) in Proposition 3.2.  Hence we have 
$$
 \begin{aligned}
  G_Y ( z ) & = \sfrac{z}{\, m \,} \, 
              G_{X_{1,m}} \Big( \sfrac{z^2}{\, m \,} \Big) 
              = \frac{z}{\, m \,} \, 
                \frac{ (m + 1) \sfrac{z^2}{\, m \,} 
                      - \sqrt{(m - 1)^2 \Big(\sfrac{z^2}{\, m \,}\Big)^2 
                      - 4 m \sfrac{z^2}{\, m \,}}}
                {2 \Big(\sfrac{z^2}{\, m \,} \Big)  
                \Big\{ 1 + \Big( \sfrac{z^2}{\, m \,} \Big) \Big\} } \\
            & = \frac{ \Big( \sfrac{m - 1}{\, m \,} \Big) z 
                   - \sqrt{\Big( \sfrac{m - 1}{\, m \,} \Big)^2 z^2 - 4}}
                     {2  \Big(1 + \sfrac{z^2}{\, m \,} \Big)} 
              = \frac{ (m + 1) z - \sqrt{(m - 1)^2 z^2 - 4 m^2  }}
                     {2 (m + z^2) } .
 \end{aligned}
$$
\end{proof} 

\medskip 

Applying Stieltjes inversion formula to $G_Y(z)$, we can easily obtain 
the probability measure of the random variable $Y$, which is absolutely 
continuous with respect to Lebesgue measure. 
This measure is our desired free $T$-distribution.

\smallskip 

\begin{definition}{\rm 
For $m > 1$, the free $T$-distribution $\mathrm{f}T(m)$ of 
parameter $m$ is the compactly supported probability measure on 
$\big[ - \sdfrac{2 m}{m - 1}, \sdfrac{2 m}{m - 1} \big]$
with the density 
$$
  f_{\mathrm{f}T} \big( x; m \big) = 
   \frac{ \sqrt{4 - \Big( \sfrac{m - 1}{\, m \,} \Big)^2 x^2}}
               {2 \pi \Big(1 + \sfrac{x^2}{\, m \,} \Big)}.
$$
}
\end{definition}

\smallskip 

\begin{remark}{\rm
The density function of the free $T$-distribution $\mathrm{f}T(m)$ 
has the following limits:
$$
\begin{aligned}
   \lim_{m \to \infty} f_{\mathrm{f}T} \big( x; m \big) 
   & = \sfrac{\sqrt{4 - x^2}}{2 \pi} & 
       \; & \mbox{(the standard semicircle law)},  \\
   \lim_{m \to \, 1 \; } f_{\mathrm{f}T} \big( x; m \big)
   & = \sfrac{1}{ \pi \big(1 + x^2 \big)} & 
       \; & \mbox{(the Cauchy distribution)}.
\end{aligned}
$$

In classical probability, it is known that the $T$-distribution $T(m)$ 
of the parameter $m$ becomes the standard Gaussian in the limit 
$m \to \infty$, and the Cauchy distribution can be obtained as the 
special case of $m=1$.

In the limit $m \to \infty$, the density of the free $T$-distribution 
$\mathrm{f}T(m)$ tends to the standard semicircle law, 
the free counterpart of the classical standard Gaussian. Since it is known 
from the theory of free stable laws \cite{BT99} that the free counterpart 
of the classical Cauchy distribution is given by the Cauchy distribution 
itself, hence, we can state that the density of the free $T$-distribution 
$\mathrm{f}T(m)$ becomes the free Cauchy distribution when $m$ goes to $1$.
Thus the above limits are compatible with those in classical probability.
}
\end{remark}

\smallskip

\subsection{The free beta distribution}
\label{sec:3_4}

In classical probability, it is well known that if $Y$ is distributed 
according to the beta distribution ${\bm \beta}(a,b)$ then 
$\sfrac{Y}{1-Y}$ is the beta prime 
${\bm \beta}^{\prime}(a, b)$-distributed random variable, or equivalently 
that if $X$ is distributed according to the beta prime distribution 
${\bm \beta}^{\prime}(a,b)$ then $\sfrac{X}{1 + X}$ is the 
beta ${\bm \beta}(a, b)$-distributed random variable.

\smallskip 

Based on this fact, it is natural to introduce the free beta 
distribution as follows:
Let $X$ be a self-adjoint random variable in a $C^*$-probability 
space distributed according to the free beta prime distribution 
$\mathrm{f}{\bm \beta}^{\prime}(a,b)$. Then we will regard the 
distribution of the reciprocal 
$B = \big( \bm{1} + X^{-1} \big)^{-1}$, as the free beta distribution. 

\smallskip 

For simplicity, we first deal with the case of $a > 1$ and $b > 1$ so that no 
atomic parts appear in 
$\mathrm{f}{\bm \beta}^{\prime}(a,b)$ or 
$\mathrm{f}{\bm \beta}^{\prime}(b,a)$.

\medskip


\begin{proposition}
We assume $a > 1$ and $b > 1$, and let $X$ be a free beta 
prime $\mathrm{f}{\bm \beta}^{\prime}(a,b)$-distributed 
self-adjoint random variable in a $C^*$-probability space. 
Then the Cauchy transform of the distribution of 
$B = \big( \bm{1} + X^{-1} \big)^{-1}$ is given by 
$$
 G_{B} (z) 
 = \frac{ (a + b - 2) z + (1 - a) - 
     \sqrt{(a + b)^2 z^2 - 2 (a b + a^2 -a + b) z + (a -1)^2}}
          {2 z (1 - z)},
\eqno{(4)}
$$
where the branch of the square root is chosen so that 
$\displaystyle{\lim_{z \to \infty} G_{B} (z) = 0}$.
\end{proposition}

\begin{proof}
We note that $X^{-1}$ is a free beta prime 
$\mathrm{f}{\bm \beta}^{\prime}(b,a)$-distributed random variable. 
Hence the Cauchy transform $G_{X^{-1}} (z)$ of the distribution of 
${X^{-1}}$ is given by the formula (1) in Proposition 3.2 
with exchanging $a$ with $b$, that is, 
$$
  G_{X^{-1}} (z) = G_{X(b, a)}(z).
$$

On the other hand, we easily find (see, for instance, 
\cite{BS10}) that if $W$ is a strictly positive random 
variable in a $C^*$-probability space with compact support, 
then $W$ is invertible and the Cauchy transform of $W^{-1}$ is 
given by the formula 
$$
   G_{W^{-1}} (z) = \frac{1}{\, z \,} 
      - \frac{1}{z^2} \,  G_W \Big( \frac{1}{\, z \,} \Big).
$$

\medskip 

Now we apply this formula to the random variable 
$W = \bm{1} + X^{-1}$ with 
$G_W (z) = G_{X(b, a)}(z - 1)$ since $W$ is the shift 
of $X^{-1}$ by $\bm{1}$. 

Combining the formulas above, we can have 
$$
 G_{B} (z) = \frac{1}{\, z \,} 
   - \frac{1}{z^2} \,  G_{X(b, a)} \Big( \frac{1}{\, z \,} - 1 \Big),
$$
which yields our desired formula. 
\end{proof}

\medskip 

We introduce the free beta distribution by the Cauchy 
transform $G_B (z)$. Although we have derived $G_B (z)$ under 
the condition $a > 1$ and $b > 1$, it can be found 
that $G_B (z)$ is still 
valid for $\big\{ (a, b) \, |  \,  
a >0, \, b > 0, \,\mbox{and} \, a + b > 1 \big\}$.

\smallskip 

Applying Stieltjes inversion formula, one can recover the probability 
measure as follows: 
Concerning the point mass of the measure, the function $G_{B} (z)$ 
has the simple poles at $z = 0$ and $z = 1$, and the residues 
are calculated as 
$$
   {\mathrm {Res}} \big( G_{B} (z) ; z = 0 \big) = \max (1 - a,0) 
  \; \mbox{ and } \; 
   {\mathrm {Res}} \big( G_{B} (z) ; z = 1 \big) = \max (1 - b,0),
$$
respectively. Thus the measure has a point masses $(1 - a)$ at 
$0$ if $0 < a < 1$ and $(1 - b)$ at $1$ if $0 < b < 1$.

\smallskip 

The density $f_{\mbox{\scriptsize f}{\bm \beta}}(x ; a, b)$ 
of the Lebesgue absolutely continuous part is given by 
$$
 \begin{aligned}
  f_{\mbox{\scriptsize f}{\bm \beta}}(x ; a, b) 
   = &  - \lim_{\varepsilon \to +0} \frac{1}{\, \pi \,} 
   \mathrm{Im} \big( G_{B} ( x + \varepsilon \mathrm{i}) \big) \\
   = & \frac{\sqrt{ - (a + b)^2 x^2 + 2 (a^2 + a b - a + b) x - (a -1)^2 }}
            {2 \pi x (1 - x)}, 
 \end{aligned}
$$
which is supported on the interval $[\kappa_{-}, \kappa_{+}]$, 
where 
$\displaystyle{ \kappa_{\pm} 
        = \bigg( \sfrac{\sqrt{a (a + b - 1) } \pm 
         \sqrt{b}}{a + b} \bigg)^{\!2}
}$ 
are two real roots of the quadratic equation
$ (a + b)^2 x^2 - 2 (a^2 + a b - a + b) x + (a -1)^2 = 0$ and 
satisfy $[\kappa_{-}, \kappa_{+}] \subseteq [0,1]$.

\smallskip 

\begin{definition}{\rm 
Let the parameters $a$ and $b$ satisfy $a >0$, $b > 0$, 
and $a + b > 1$. The free beta distribution 
${\mathrm{f}{\bm \beta}}(a, b)$ is the probability measure 
$$
 \begin{aligned}
  d \upsilon_{a,b} (x) = & 
       \frac{(a + b) \sqrt{- (x - \kappa_{+}) (x - \kappa_{-})} }
       {2 \pi x (1 - x)} \, I_{[\kappa_{-}, \kappa_{+}]}(x) \, dx \\
       & \qquad \qquad \quad + \max(1 - a, 0)\, \delta_0 (x)
                      + \max(1 - b, 0)\, \delta_1 (x),
 \end{aligned}
$$
where $\displaystyle{ \kappa_{\pm} 
        = \bigg( \sfrac{\sqrt{a (a + b -1)} \pm 
         \sqrt{b}}{a + b} \bigg)^{\!2}}$.
}
\end{definition}

\medskip

%
%
\section{The rationales of the free analogue}
\label{sec:4}

\subsection{The first rationale of the free analogue}
\label{sec:4_1}

The standard semi-circular distribution can be regarded as the 
free analogue of the standard normal distribution, and the 
square of a standard semi-circular element has the distribution of 
the free Poisson of parameter $1$.
Hence it is natural to regard the probability distribution 
of the freely independent sum of $m$-many squares of standard 
semi-circular elements as the free $\chi^2$-distribution of 
degree of freedom $m$, which is given by the free Poisson 
distribution of parameter $m$.

\smallskip

In classical probability, the $\chi^2$-distribution is in a class 
of gamma distributions, and the ratio of independent gamma distributed 
random variables gives the beta prime distribution. Indeed, 
if $X_1$ and $X_2$ are independent and distributed according to the 
gamma distributions $\Gamma(a, \theta)$ and $\Gamma(b, \theta)$, 
respectively, where both have the same scaling parameter $\theta$, 
then the random variable $X_1/X_2$ has the beta prime distribution 
${\bm \beta}^{\prime}(a, b)$.

\smallskip

From this standpoint, the free beta prime distribution introduced 
in the previous section can be regarded as the distribution of a 
ratio of freely independent free gamma distributed random variables 
because the free Poisson random variables have the free gamma 
distributions in the sense of free $\chi^2$-distributions.
This is the first naive rationale of our free analogue of the 
beta prime distribution.

\subsection{The score functions and the potentials}
\label{sec:4_2}
The second rationale is related to the score functions for the 
Fisher informations, in other words, the potentials of the diffusion 
processes. Here we will briefly recall the score functions both 
in the classical and the free cases.

\medskip 

In classical probability, for a probability measure $\mu$ with 
the differentiable density $f$, the function 
$$
   \rho_\mu (x) 
 = \dfrac{d}{d x} \big( \log f(x) \big) = \dfrac{f'(x)}{f(x)}
$$
is called the score function of $\mu$. 
Then the classical Fisher information $I(\mu)$ of $\mu$ (with 
respect to the location parameter) can be given by the square 
of the $L^2$-norm of the score function $||\rho_\mu||^2$ in 
$L^2 (d \mu)$, that is, 
$$
  I(\mu) = \int \Big( \sfrac{f'(x)}{f(x)} \Big)^2 \,d \mu(x)
         = \int \rho_\mu (x)^2 \,d \mu(x).
$$

One of the basic properties of the score function is the following 
Stein's relation: For a smooth function $p$, applying integration 
by parts, the relation 
$$
    \int p(x) \rho_\mu (x) \, d \mu (x)
  = - \int p'(x) \, d \mu (x).
$$
holds.

\medskip

On the other hand, the free analogues of entropy and Fisher 
information for self-adjoint non-commutative random variables 
were introduced and begun to study in Voiculescu's paper 
\cite{Vo93} (see, for survey, \cite{Vo02}). 
In the univariate case, 
the free Fisher information $\Phi(\mu)$ of a compactly supported 
probability measure $\mu$ is given by 
$$
  \Phi(\mu) = \int \! \Big( 2 \, \hilb{f}(x) \Big)^2 \, d \mu(x), 
$$
where $\hilb{f}(x)$ is the ($\pi$-multiplicated) Hilbert transform 
of $f$ defined by the principal value integral 
$$
  \hilb{f}(x) = {\mathrm{p.v.}} 
                \! \int \! \frac{f (y) }{ x - y } \, dy 
  = \lim_{\varepsilon \to 0} 
           \bigg(  \int_{-\infty}^{x - \varepsilon} 
           \! \! + \! 
           \int_{x + \varepsilon}^{\infty}
           \frac{f(y)}{x - y} \, dy \bigg).
$$

By the expression of the free Fisher information, it is natural to 
understand that the double of the Hilbert transform $2 \hilb{f} (x)$ 
corresponds to a free analogue of the classical score function, 
because the square of the $L^2$-norm 
$|| 2 \hilb{f} ||^2$ in $L^2 (d \mu)$ 
is the free Fisher information of $\mu$. 

Indeed the function $2 \hilb{f} (x)$ satisfies the following identity
(see, for instance, \cite{NY14}):
Let $\mu$ be a compactly supported probability 
measure on ${\mathbb R}$ with continuous density $f$, and assume 
that $\mu$ has finite free Fisher information. 
Then, for a continuously differentiable function $\eta$ on 
${\mathbb R}$, we have  
$$
     \int_{\mathcal{S}} \! 
      \eta(x) \Big( 2 \, \hilb{f}(x) \Big) \, d \mu(x)
   = \iint_{{\mathcal{S}} \times {\mathcal{S}}} \! \! \! 
           \frac{\eta(x) - \eta(y)}{x - y} \, d \mu(x) \, d \mu(y),
$$
where ${\mathcal{S}}={\mathrm{Supp}}(\mu)$, which can be regarded as 
the free counterpart of the classical Stein's relation because 
the difference quotient $D \eta = \sfrac{\eta(x) - \eta(y)}{x - y}$
works as non-commutative derivative.

\medskip 

Comparing the free Stein's relation with the classical one, one can 
find that the sign of the free score function is opposite to 
the classical one, which is, however, compatible from the viewpoint 
of the potentials in diffusion processes.

\smallskip

For the function $V \in C^1 ({\mathbb R})$, 
$$
  g(x) = \sdfrac{1}{\mathcal{Z}} \, \exp \big( - V(x) \big), 
$$
is called the Gibbs distribution of the potential $V$, 
which is obtained as the long-time asymptotically stationary 
distribution for the diffusion process on ${\mathbb{R}}$ 
with the drift potential $V$. It is clear that the classical score 
function of the Gibbs distribution $g$ is given 
by $\dfrac{g'(x)}{g(x)} = - V'(x)$.

\smallskip

On the other hand, in the free context, Biane and Speicher 
 (see \cite{Bi03} and \cite{BiS01}) investigated the free analogue of 
diffusion process via random matrix models. They derived 
that the long-time asymptotically stationary 
measure $\nusV$ for the free diffusion process with the potential 
$V \in C^1 ({\mathbb R})$ is characterized by the Euler-Lagrange equation
$$
   \big( {\mathcal H} g \big) (x) = \onehalf V' (x) \; \;
   \mbox{ on } {\mathrm{Supp}}(\nusV),
$$
where $g$ is the compactly supported density $g(x) dx = d \nusV (x)$.

This stationary measure $\nusV$ is called the equilibrium 
measure for the free diffusion process with the potential $V$. 
It is obvious that the free score function of the equilibrium 
measure $\nusV$ is given by 
$2 \big({\mathcal H} g \big) (x) = V'(x)$ just like with 
the Gibbs distribution.

\smallskip

The classical and the free score functions mentioned above, 
namely, the derivative of the potential $V'(x)$ gives the
second rationale of our free analogues.

\subsection{The second rationale of the free analogue}
\label{sec:4_3}

We shall see that the free beta prime distribution and the free 
analogues of the classical $T$ and beta distributions derived in 
the previous section can be characterized by exactly the 
same (up to constant) potentials for the Gibbs forms of 
corresponding classical distributions.

\smallskip

The following formula is helpful for us to find the 
($\pi$-multiplicated) Hilbert transform of the probability 
measure on ${\mathbb{R}}$ (see, for instance, Chapter 3 in \cite{HP00}): 
For a compactly supported probability measure $\mu$ on ${\mathbb{R}}$, 
the ($\pi$-multiplicated) Hilbert transform $\hilb{\mu}$ can be 
obtained by the formula, 
$$
   \hilb{\mu}(x) = 
   \lim_{\varepsilon \to +0} 
   {\mathrm {Re}} \big( G_{\mu}(x + \varepsilon {\mathrm i}) \big),
$$
where $G_\mu$ is the Cauchy transform of $\mu$.

\medskip 

We shall list the corresponding potentials of our free analogues below.

\noindent
(i) {\it The free beta prime distribution}:
\begin{itemize}
\item[$\bullet$]{
The density of the classical beta prime distribution 
${\bm \beta}^\prime (a, b)$ and its potential: 
$$
 \begin{aligned}
      & \sfrac{1}{B(a,b)} \, x^{a - 1} (1 + x)^{- a - b} \\
    = & \sfrac{1}{B(a,b)} \,
      \exp \Big\{ -\big((1 - a) \log x + (a + b) \log (1 + x) \big) \Big\}
      = \exp \big( - V_{{\bm \beta}^\prime (a, b)} (x) \big)
 \end{aligned}
$$
for $x > 0$, where ${B}$ is the beta function.  }
\item[$\bullet$]{
The derivative of the potential: 
$$
 V_{{\bm \beta}^\prime (a, b)}^{\prime} (x) 
    = \frac{1 - a}{x} + \frac{a + b}{1 + x}
           = \frac{(b + 1) x + (1 - a)}{x (1 + x)}.
$$
}
\item[$\bullet$]{
The free score function of the free beta prime distribution
${\mathrm{f}{\bm \beta}^{\prime} (a,b)}$: \\
\indent
The Cauchy transform of ${\mathrm{f}{\bm \beta}^{\prime} (a,b)}$ 
is given by 
(1) in Proposition 3.2 
and its free score function is 
$$
  2 \hilb{f_{\mathrm{f}{\bm \beta}^{\prime} (a,b)}}(x)
   = \frac{(b + 1) x + (1 - a)}{x (1 + x)} 
   = V_{{\bm \beta}^\prime (a, b)}^{\prime} (x).
$$
}
\end{itemize}
\noindent
(ii) {\it The free $T$-distribution}:
\begin{itemize}
\item[$\bullet$]{
The density of the classical $T$-distribution $T(m)$ and its potential: 
$$
 \begin{aligned}
   & \sfrac{1}{\sqrt{m} B \big( \frac{1}{\, 2 \,}, \frac{m}{\, 2 \,} \big)} 
     \Big( \,  1 + \sfrac{x^2}{\, m \,} \,  \Big)^{-\frac{m + 1}{2}} \\
 = & \sfrac{1}{\sqrt{m} B \big( \frac{1}{\, 2 \,}, \frac{m}{\, 2 \,} \big)} 
     \exp \Big\{ - \Big( \sfrac{m + 1}{2} \Big) 
     \log \Big( \, 1 + \sfrac{x^2}{\, m \,} \, \Big) \Big\}
 = \exp \big( - V_{T(m)} (x) \big).
 \end{aligned}
$$
}
\item[$\bullet$]{
The derivative of the potential: 
$$
  V_{T(m)}^{\prime} (x) = \Big( \sfrac{m + 1}{2} \Big) \cdot 
                     \Big(  1 + \sfrac{x^2}{\, m \,} \Big)^{\!-1} \cdot 
                     \Big( \sfrac{2 x}{m} \Big) 
                   = \frac{(m + 1) x}{ m + x^2}.
$$
}
\item[$\bullet$]{
The free score function of the free $T$-distribution
${\mathrm{f} T(m) }$: \\
\indent
The Cauchy transform of ${\mathrm{f} T(m)}$ is given by 
(3) in Proposition 3.6 
and its free score function is 
$$
  2 \hilb{f_{\mathrm{f}T(m)}}(x)
   = \frac{(m + 1) x}{ m + x^2}
   = V_{T(m)}^{\prime} (x).
$$
}
\end{itemize}
\noindent
(iii) {\it The free beta distribution}:
\begin{itemize}
\item[$\bullet$]{
The density of the classical beta distribution ${\bm \beta}(a, b)$ and 
its potential: 
$$
 \begin{aligned}
      & \sfrac{1}{B(a,b)} \, x^{a - 1} (1 - x)^{b - 1} \\
    = & \sfrac{1}{B(a,b)} \,
      \exp \Big\{ -\big((1 - a) \log x + (1 - b) \log (1 - x) \big) \Big\}
      = \exp \big( - V_{{\bm \beta}^\prime (a, b)} (x) \big)
 \end{aligned}
$$
for $0 < x < 1$, where ${B}$ is the beta function. 
}
\item[$\bullet$]{
The derivative of the potential: 
$$
 V_{{\bm \beta} (a, b)}^{\prime} (x) 
    = \frac{1 - a}{x} - \frac{1 - b}{1 - x}
           = \frac{(a + b - 2) x + (1 - a)}{x (1 - x)}.
$$
}
\item[$\bullet$]{
The free score function of the free beta distribution 
${\mathrm{f}{\bm \beta} (a,b)}$: \\
\indent
The Cauchy transform of ${\mathrm{f}{\bm \beta} (a,b)}$ is given by 
(4) in Proposition 3.9 
and its free score function is 
$$
  2 \hilb{f_{\mathrm{f}{\bm \beta} (a,b)}}(x)
   = \frac{(a + b - 2) x + (1 - a)}{x (1 - x)} 
   = V_{{\bm \beta}(a, b)}^{\prime} (x).
$$
}
\end{itemize}

\begin{remark}{\rm
Concerning the second rationale, it should be noted that 
Hasebe and Szpojankowski pointed out such a correspondence 
between the classical and the free distributions in \cite{HS18}
based on the maximization problem of the entropy functionals 
with an external potential $V$.

In particular, they showed the correspondence between 
the classical and the free generalized inverse Gaussian 
distributions. They also mentioned the maps from the classical 
Gaussian to the semicircle and from the classical gamma to the 
free Poisson. 
Hence the free analogue of the classical distributions derived 
in the previous section can be appended as new examples.
}
\end{remark}

\medskip

%
%
\section{The combinatorial representation of the moments}
\label{sec:5}

\subsection{Non-crossing linked partitions and the Motzkin paths}
\label{sec:5_1}

In order to describe the moments in combinatorial way, we shall 
use the set partitions.
For the set $[n] = \{ 1,2, \ldots , n \}$, a partition of 
$[n]$ is a collection $\pi =$ $\{ B_1, B_2,$ $\ldots ,$ $ B_k \}$ 
of non-empty disjoint subsets of $[n]$ which are called blocks
and whose union is $[n]$. 
For a block $B$, we denote by $|B|$ the size of the block $B$, 
that is, the number of the elements in the block $B$. 
A block $B$ will be called singleton if $|B|=1$. 

\smallskip 

We say two blocks $B_i$ and $B_j$ in $\pi$ are crossing if 
there exist elements $b_1, b_2 \in B_i$, $c_1, c_2 \in B_j$ 
such that $b_1 < c_1 < b_2 < c_2$. The blocks $B_i$ and $B_j$ are 
said to be non-crossing if they are not crossing. 
A partition $\pi$ is called non-crossing if its blocks 
are pairwise non-crossing. We denote the set of all non-crossing 
partitions of the set $[n]$ by $\mathcal{NC} (n)$. The notion of 
non-crossing partition was first introduced in \cite{Kr72}.  
For more about non-crossing partitions, see the survey of Simion 
\cite{Si00}. 

\smallskip 

In the context of free probability, Dykema introduced a new 
structure of partitions, the non-crossing linked partitions 
in \cite{Dy07} (see also \cite{Ni10}), which can be regarded as 
a non-crossing partition having some links between blocks with 
certain restrictions. The restricted 
link between blocks introduced by Dykema in \cite{Dy07} 
is as follows:
\begin{itemize}
\item[]{
Let $E$ and $F$ be subsets of $[n]$, We say that 
$E$ and $F$ are nearly disjoint if for every $i \in E \cap F$, 
one of the following holds:
$$
 \begin{aligned}
  \mbox{(a)} \; & i = \min(E), \, |E| > 1 \; \mbox{and} \; i \ne \min(F), 
        \qquad \qquad \qquad \qquad \qquad \qquad \qquad \qquad \qquad  \\
  \mbox{(b)} \; & i \ne \min(E), \, i = \min(F) \; \mbox{and} \; |F| > 1. 
 \end{aligned}
$$
}
\end{itemize}

\smallskip 

He derived the structure of non-crossing linked partitions 
in his study on the multiplicative free convolution and 
the $T$-transform.

\begin{definition}{\rm 
A non-crossing linked partition of $[n] = \{ 1, 2, \ldots, n \}$ 
is a collection $\pi$ of non-empty subsets of $[n]$ whose union 
is $[n]$, and any two distinct elements of $\pi$ are non-crossing and 
nearly disjoint. We denote by $\mathcal{NCL}(n)$ the set of 
all non-crossing linked partitions of $[n]$. 

For $\pi \in \mathcal{NCL}(n)$, although elements of $\pi$ are not 
disjoint in general, we refer to an element of $\pi$ as a 
block of $\pi$.
}
\end{definition}

Here, we recall some terminologies and the basic properties  
of non-crossing linked partitions observed in \cite{Dy07}.

\begin{remark}{\rm
\noindent
\vspace{-5pt}
\begin{itemize}
\item[]{
  \begin{itemize}
    \item[(1)] 
    {Any given $k \in [n]$ belongs to either exactly one block or 
     exactly two blocks; we will say $k$ is singly or 
     doubly covered by $\pi$, accordingly.}
    \item[(2)]
     {The elements $1$ and $n$ are singly covered by $\pi$.}
    \item[(3)]
     {Any two distinct elements $E$ and $F$ of $\pi$ have at most one 
      element in common. Moreover, if $E \cap F \ne \phi$ 
    (that is, $| E \cap F | = 1$), then both $E$ and $F$ have at least
     two elements.}
  \end{itemize}
}
\end{itemize}
}
\end{remark}

\smallskip

One of graphical representations of non-crossing linked partitions has 
been described in \cite{Dy07}, which is a modification of the usual 
pictures of non-crossing partitions in the following way:
The non-crossing partitions 
$\pi \in \mathcal{NC}(n) \subseteq \mathcal{NCL}(n)$ are drawn in the usual 
way with all angles being right angles. 
Suppose $\pi \in \mathcal{NCL}(n) \! \setminus \! \mathcal{NC}(n)$. 
If $E, F \in \pi$ with $E \ne F$, and if $E \cap F = {\min(E)}$, then the 
line connecting $\min(E)$ to the next element in $E$ is started with 
diagonal line, that is, the diagonally started lines indicate the links.

\begin{example}{\rm
 The non-crossing linked partition 
\vspace{-5pt}
$$
  \pi = \big\{ \{ 1, 2, 7 \}, 
               \{ 2, 4 \}, 
               \{ 3 \}, 
               \{ 5, 6 \}, 
               \{ 7, 8, 9 \}, 
               \{ 9, 10 \} \big\}
\vspace{-5pt}
$$
has the graphical representation,
$$
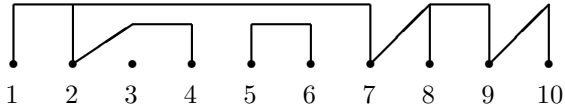
\setlength{\unitlength}{1.5pt}
\begin{picture}(170, 30)(0,0)
  \multiput(15,10)(15,0){10}{\circle*{2}}
  \put( 13, 0){\small  $1$}
  \put( 28, 0){\small  $2$}
  \put( 43, 0){\small  $3$}
  \put( 58, 0){\small  $4$}
  \put( 73, 0){\small  $5$}
  \put( 88, 0){\small  $6$}
  \put(103, 0){\small  $7$}
  \put(118, 0){\small  $8$}
  \put(133, 0){\small  $9$}
  \put(147, 0){\small $10$}
  \thicklines
  \put( 15, 10){\line(0,1){15}}
  \put( 30, 10){\line(0,1){15}}
  \put(105, 10){\line(0,1){15}}
  \put( 15, 25){\line(1,0){90}}
%
  \put( 30, 10){\line(3,2){15}}
  \put( 45, 20){\line(1,0){15}}
  \put( 60, 10){\line(0,1){10}}
%
  \put( 75, 10){\line(0,1){10}}
  \put( 90, 10){\line(0,1){10}}
  \put( 75, 20){\line(1,0){15}}
%
  \put(105, 10){\line(1,1){15}}
  \put(120, 10){\line(0,1){15}}
  \put(135, 10){\line(0,1){15}}
  \put(120, 25){\line(1,0){15}}
%
  \put(135, 10){\line(1,1){15}}
  \put(150, 10){\line(0,1){15}}
%
\end{picture}
$$
There are three doubly covered elements, $2$, $7$, and $9$.
}
\end{example}

\smallskip

In \cite{CWY08}, Chen, Wu, and Yan proposed another graphical representation 
of non-crossing linked partitions, called the linear representation, 
which is defined as follows:  Given a non-crossing linked partition $\pi$ of 
$[n]$, list $n$ vertices in a horizontal line with labels $1$, $2$, $\ldots$, 
$n$. For each block $E = \{i_1, i_2, \ldots, i_\ell \}$ of $\pi$ with 
$i_1 = \min (E)$ and $\ell \ge 2$, draw an arc between $i_1$ and $i_j$ for each 
$j = 2, \ldots , \ell$, where we should always put the arc $(i, j)$ above 
the arc $(i, k)$ if $j > k$. 

Using the linear representation, they constructed a bijection between 
non-crossing linked partitions and Schr{\"o}der paths and derived various 
enumerative results on non-crossing linked partitions 
(see, for details, \cite{CWY08}). 

\medskip 

Here we shall introduce another graphical representation of 
non-crossing linked partitions, namely, the card arrangements 
which is essentially the same as the above graphical representation 
of Dykema, but we should much more consider the heights of lines 
in order to reveal the relation between non-crossing linked 
partitions and the Motzkin paths.

\begin{definition}{\rm 
A Motzkin path of length $n$ is a non-negative lattice path 
from $(0, 0)$ to $(n, 0)$ in the integer lattice 
${\mathbb Z}_{\ge 0} \times {\mathbb Z}_{\ge 0}$ 
consisting of three types of steps:

\smallskip 

\quad 
  $u = (1,  1)$:  {\it up} step, \; 
  $d = (1, -1)$:  {\it down} step, \; 
  $t = (1,  0)$:  {\it transit} step.
}
\end{definition}

The set partitions are closely related to the Motzkin paths. 
Indeed Flajolet investigated the correspondence between all 
partitions of $n$ elements and colored (integer labeled) Motzkin 
paths of length $n$ in \cite{Fl80}.

\smallskip 

In the paper \cite{YY07}, the representation of non-crossing partitions 
by cards arrangements associated with the Motzkin paths was shown, 
which was the similar technique to the juggling patterns in \cite{ER96} but 
they were required to prepare different kinds of cards. 
Namely, they used {\it the opening}, {\it the closing}, 
{\it the intermediate}, and {\it the singleton} cards in order to 
represent non-crossing partitions. 
We list these cards below and illustrate the representation of 
non-crossing partitions by cards again for our convenience. 

\medskip

\noindent
{\it The opening cards}: 

The opening card $O_i \; (i = 0,1,2, \ldots)$ has $i$ inflow lines from 
the left and $(i\!+\!1)$ outflow lines to the right, where one new line 
starts from the middle point on the ground level. 
For each $j \ge 1$, the inflow line of the $j$th level goes through out 
to the $(j\!+\!1)$st level without any crossing. 
The card $O_i$ is called {\it the opening card of level} $i$.
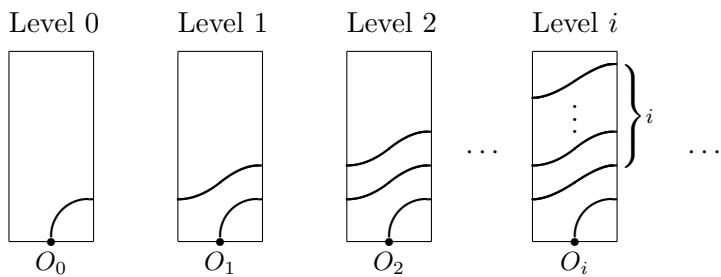
$$
 \begin{array}{llllll}
    \quad \mbox{Level } 0 & \quad \mbox{Level } 1 & \quad \mbox{Level } 2 & 
     & \mbox{Level } i & \\
       \quad \opnzer{}{O_0}{} \quad 
     & \quad \opnone{}{O_1}{} \quad 
     & \quad \opntwo{}{O_2}{} 
     & \dumcdots
     &       \opnnth{}{O_i}{} \rsideihibig{i}
     & \dumcdots
 \end{array} 
$$
The opening card represents the minimal element of a block of 
non-singleton. 

\medskip

\noindent
{\it The closing cards} :

The closing card $C_i \; (i = 1,2,3, \ldots)$ has $i$ inflow 
lines from the left and $(i\!-\!1)$ outflow lines to the right.  
On the card $C_i$, only the line of the lowest level goes 
down to the middle point on the ground level and ends.
For each $j \ge 2$, the inflow line of the $j$th level goes 
through out to the $(j\!-\!1)$st level without any crossing. 
The card $C_i$ is called {\it the closing card of level} $i$.
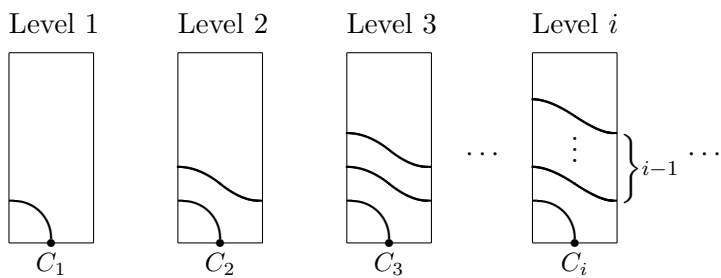
$$
 \begin{array}{llllll}
    \quad \mbox{Level } 1 & \quad \mbox{Level } 2 & \quad \mbox{Level } 3 &   
     &  \mbox{Level } i &  \\
       \quad \clsone{}{C_1}{} \quad 
     & \quad \clstwo{}{C_2}{} \quad 
     & \quad \clsthr{}{C_3}{} 
     & \dumcdots     
     &       \clsnth{}{C_i}{} \rsideilo{i-1}
     & \dumcdots 
 \end{array}
$$
The closing card represents the maximal element in a 
block of non-singleton. 

\medskip

\noindent
{\it The intermediate cards} : 

The intermediate card $I_i \; (i = 1,2,3, \ldots)$ has $i$ inflow 
lines and the same number of outflow lines.  Only the line of the 
lowest level goes down to the middle point on the ground and continues 
its flow as the lowest line again. The rest of inflow lines maintain
their levels.  
We call the card $I_i$ {\it the intermediate card of level} $i$.
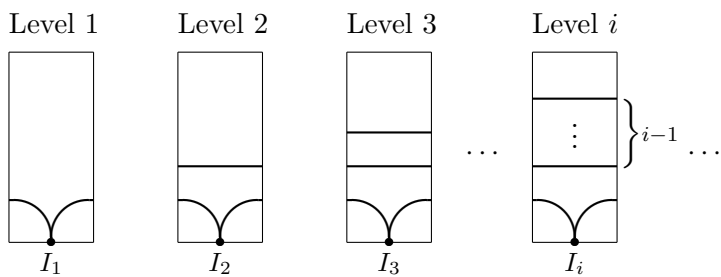
$$
 \begin{array}{llllll}
    \quad \mbox{Level } 1 & \quad \mbox{Level } 2 & \quad \mbox{Level } 3 & 
     & \mbox{Level } i & \\
       \quad \imdone{}{I_1}{} \quad 
     & \quad \imdtwo{}{I_2}{} \quad 
     & \quad \imdthr{}{I_3}{} 
     & \dumcdots
     &       \imdnth{}{I_i}{} \rsideihi{i-1}
     & \dumcdots 
 \end{array}
$$
The intermediate card represents the intermediate (neither the 
minimal nor the maximal) element of a block of size $\ge 3$.

\medskip

\noindent
{\it The singleton cards} :

The singleton card $S_i \; (i = 0,1,2, \ldots)$ has $i$ 
horizontally parallel lines and the short pole at the middle point 
on the ground.  
We call the card $S_i$ {\it the singleton card of level} $i$, which 
represents, of course, a singleton.
$$
 \begin{array}{llllll}
    \quad \mbox{Level } 0 & \quad \mbox{Level } 1 & \quad \mbox{Level } 2 & 
     & \mbox{Level } i &  \\
       \quad \sglzer{}{S_0}{} \quad 
     & \quad \sglone{}{S_1}{} \quad 
     & \quad \sgltwo{}{S_2}{} 
     & \dumcdots 
     &       \sglnth{}{S_i}{} \rsideilo{i}
     & \dumcdots
 \end{array}
$$

Associated with a Motzkin path, we arrange the above cards according to the 
following rule:

\smallskip 

\noindent
{\it The rule of the arrangements for ${\mathcal{NC}}(n)$} 

Let $\mathbf{p} = \big( s_1, s_2, ... , s_n \big)$ be a Motzkin path  
of length $n$ where $s_j \in \{ u, d, t \}$, and denote by $y_j$ 
the height of the step $s_j$ starting.
\begin{itemize}
\setlength{\leftskip}{7mm}
\item[(1)]{ 
In case of $s_j = u$, {\it up} step (resp. $s_j = d$, {\it down} step), 
we put the opening (resp. closing) card of level $y_j$ at the 
$j$th site. 
}
\item[(2)]{ 
In case of $s_j = t$, {\it transit} step, if the height $y_j \ge 1$ 
(not at ground level) then two cards are available, namely, the 
intermediate card of level $y_j$ or the singleton card of level $y_j$,
but if $y_j = 0$ (at the ground level) then we have to put 
$S_0$, the singleton card of level $0$, at the $j$th site.
}
\end{itemize}

\smallskip 

The card arrangements constructed by the rule above are called 
the admissible arrangements. Each Motzkin path yields not only one 
admissible arrangement in general but each admissible arrangement 
determines the non-crossing partition of $[n]$ 
uniquely, the blocks of which are constituted from the connected curves 
in the pattern on the admissible arrangement 
(see, for more details, \cite{YY07}). 

\smallskip 

Now we shall extend the above representation to the case of non-crossing 
{\it linked} partitions. To this end, we introduce some more cards 
which will represent the doubly covered elements. 

Before making the cards, we classify the doubly covered elements 
into two types.
It is from the definition that a doubly covered element is 
contained in two blocks and the minimal element of one or 
the other. 
Let $k$ be a doubly covered element such that $k \in E \cap F$ with 
$k = \min(F)$. 

\begin{itemize}
\setlength{\leftskip}{7mm}
\item[(i)] {If $k = \max(E)$, then we call $k$ {\it the doubly 
            covered element of type} I. } 
\item[(ii)] {Otherwise, in the case where $k \ne \max(E)$, that is, 
            $k$ is an intermediate element of $E$ ($k \ne \min(E)$ 
            follows by definition) then $k$ is said to be 
            {\it the doubly covered element of type} II}. 
\end{itemize}
Here we will construct the cards for the doubly covered 
elements according to these types.

\medskip 

\noindent
{\it The cards for doubly covered elements of type} I :

\smallskip

We shall make the cards $T_i \; (i = 1,2,3, \ldots)$ for 
doubly covered elements of type I. 
The card $T_i$ has $i$ inflow lines and the same number of 
outflow lines. 
Only the inflow line of the lowest level goes down to the middle point on 
the ground and ends, which indicates the end of a block. 
But immediately a new line starts with $\pi/4$ - slope from the same middle 
point on the ground level and will be the outflow line of the 
lowest level. 
The rest of inflow lines maintain their levels.  
$$
 \begin{array}{llllll}
    \quad \mbox{Level } 1 & \quad \mbox{Level } 2 & \quad \mbox{Level } 3 & 
     & \mbox{Level } i &  \\
       \quad \trsone{}{T_1}{} \quad 
     & \quad \trstwo{}{T_2}{} \quad 
     & \quad \trsthr{}{T_3}{} 
     & \dumcdots
     &       \trsnth{}{T_i}{} \rsideihi{i-1}
     & \dumcdots 
 \end{array}
$$
On the card $T_i$, the $\pi/4$ - slope indicates the beginning of a new 
block, which means a doubly covered element of type I.

\medskip 

\noindent
{\it The cards for doubly covered elements of type} II :

\smallskip

We prepare the cards $U_i \; (i = 1,2,3, \ldots)$ for doubly covered 
elements of type II. The card $U_i$ has $i$ inflow lines from  
the left and $(i\!+\!1)$ outflow lines to the right, where only the inflow 
line of the lowest level goes down to the middle point on the ground and 
continues its flow as the second (not the lowest) level again. This  
connected flow indicates an intermediate element.
Immediately a new line starts with $\pi/4$ - slope from the same middle 
point on the ground level and will be the outflow line of the lowest level, 
which indicates the beginning of a new block and, hence, the card $U_i$ 
represents a double covered element of type II. 
$$
 \begin{array}{llllll}
    \quad \mbox{Level } 1 & \quad \mbox{Level } 2 & \quad \mbox{Level } 3 & 
     & \mbox{Level } i &  \\
       \quad \updone{}{U_1}{} \quad 
     & \quad \updtwo{}{U_2}{} \quad 
     & \quad \updthr{}{U_3}{} 
     & \dumcdots 
     &       \updnth{}{U_i}{} \rsideihibig{i}
     & \dumcdots 
 \end{array}
$$

\medskip 

We shall give the representation of non-crossing {\it linked} partitions by 
the arrangements of cards, which is almost the same as for 
non-crossing partitions but more cards are available at some steps.  

\smallskip 

\noindent
{\it The rule of the arrangements for ${\mathcal{NCL}}(n)$} 

Let $\mathbf{p} = \big( s_1, s_2, ... , s_n \big) \;$ $s_j \in \{ u, d, t \}$ 
be a Motzkin path of length $n$ and denote by $y_j$ the height of the 
step $s_j$ starting. Associated with the Motzkin path, we shall arrange the 
cards along with the following rule:
\begin{itemize}
\setlength{\leftskip}{7mm}
\item[(1)]{ 
In case of $s_j = u$ ({\it up} step), if the height $y_j = 0$ 
(at ground level) 
then we put the opening card of level $0$, and if $y_j \ge 1$ (not at ground 
level) then we put the card for doubly covered elements of type II or the 
opening card of level $y_j$ at the $j$th site, that is, two cards are 
available.
}
\item[(2)]{ 
In case of $s_j = t$ ({\it transit} step), if the height $y_j \ge 1$ (not at 
ground level) then three cards are available, namely, doubly covered 
elements of type I of level $y_j$, or the intermediate card of level $y_j$, 
or the singleton card of level $y_j$, but if $y_j = 0$ (at ground level) 
then we have to put the singleton card of level $0$ at the $j$th site, that 
is, there is no possibility other than $S_0$.
}
\item[(3)]{ 
If $s_j = d$ ({\it down} step), then we put the closing card of level $y_j$
at the $j$th site, which is unique possibility. 
}
\end{itemize}

\medskip 

We again call the card arrangements constructed by the new rule above 
the admissible arrangements again.  Of course, the number of the new 
admissible arrangements are rather increased compared with the case of 
non-crossing partitions. 
Similar to the case of non-crossing partitions, each admissible arrangement 
determines the non-crossing linked partition uniquely, the blocks of which 
are constituted from the connected lines in the pattern on the admissible 
arrangement, where we regard that the lines starting with the 
$\pi/4$ - slope are not connected to curved lines at doubly covered 
elements.

Conversely, it can be said that every non-crossing linked partition 
is represented as above admissible arrangement (recall the graphical 
representation of Dykema in Example 5.3).

\medskip 

\begin{example}{\rm
If the Motzkin path $\mathbf{p} = (u,u,t,d,d)$, 
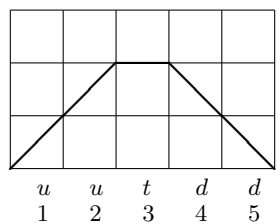
$$
\setlength{\unitlength}{1.0pt}
\begin{picture}(120,80)(0,0)
 \thinlines
  \multiput(20,20)(20,0){6}{\line(0,1){20}}
  \multiput(20,40)(20,0){6}{\line(0,1){20}}
  \multiput(20,60)(20,0){6}{\line(0,1){20}}
  \multiput(20,20)(20,0){5}{\line(1,0){20}}
  \multiput(20,40)(20,0){5}{\line(1,0){20}}
  \multiput(20,60)(20,0){5}{\line(1,0){20}}
  \multiput(20,80)(20,0){5}{\line(1,0){20}}
  \put( 30,10){\footnotesize $u$}
  \put( 50,10){\footnotesize $u$}
  \put( 70,10){\footnotesize $t$}
  \put( 90,10){\footnotesize $d$}
  \put(110,10){\footnotesize $d$}
  \put(30, 0){\footnotesize $1$}
  \put(50, 0){\footnotesize $2$}
  \put(70, 0){\footnotesize $3$}
  \put(90, 0){\footnotesize $4$}
  \put(110,0){\footnotesize $5$}
  \thicklines
  \put( 20,20){\line(1, 1){20}}
  \put( 40,40){\line(1, 1){20}}
  \put( 60,60){\line(1, 0){20}}
  \put( 80,60){\line(1,-1){20}}
  \put(100,40){\line(1,-1){20}}
 \thinlines
\end{picture}
$$
then we obtain $6$ admissible arrangements because at 
the second step $u$, we can use one of two cards $O_1$, $U_1$,
and at the third step $t$, three cards $I_2$, $S_2$, $T_2$ are 
available. We shall list the $6$ admissible arrangements and the 
corresponding non-crossing liked partitions below:
$$
  \begin{aligned}
    &(1) \; \; \pi = \big\{ \{1, 5 \}, \{ 2, 3, 4 \} \big\} 
  & &(2) \; \; \pi = \big\{ \{1, 2, 5 \}, \{ 2, 3, 4 \} \big\}  \\
    & \mbox{$ 
             \opnzer{}{1}{O_0}
             \opnone{}{2}{O_1}
             \imdtwo{}{3}{I_2}
             \clstwo{}{4}{C_2}
             \clsone{}{5}{C_1} 
            $}
    & \quad \;
    & \mbox{$ 
             \opnzer{}{1}{O_0}
             \updone{}{2}{U_1}
             \imdtwo{}{3}{I_2}
             \clstwo{}{4}{C_2}
             \clsone{}{5}{C_1}
            $}               
  \end{aligned}
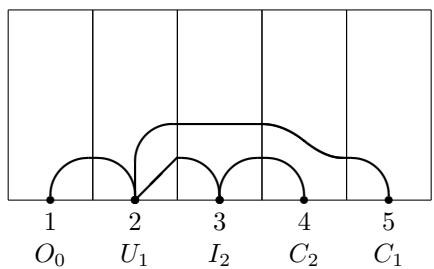
$$

$$
  \begin{aligned}
    &(3) \; \; \pi = \big\{ \{1, 5 \}, \{ 2, 4 \}, \{ 3 \} \big\}  
  & &(4) \; \; \pi = \big\{ \{1, 2, 5 \}, \{ 2, 4 \}, \{ 3 \}  \big\}      \\
    & \mbox{$ 
             \opnzer{}{1}{O_0}
             \opnone{}{2}{O_1}
             \sgltwo{}{3}{S_2}
             \clstwo{}{4}{C_2}
             \clsone{}{5}{C_1}
            $}
    & \quad \;
    & \mbox{$ 
             \opnzer{}{1}{O_0}
             \updone{}{2}{U_1}
             \sgltwo{}{3}{S_2}
             \clstwo{}{4}{C_2}
             \clsone{}{5}{C_1}
            $}
  \end{aligned}
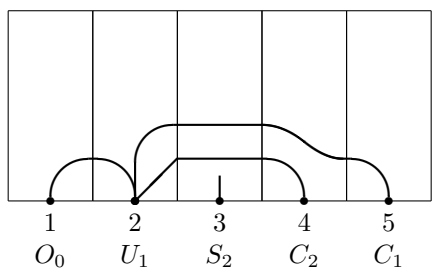
$$

$$
  \begin{aligned}
    &(5) \; \; \pi = \big\{ \{1, 5 \}, \{ 2, 3 \}, \{ 3, 4 \} \big\} 
  & &(6) \; \; \pi = \big\{ \{1, 2, 5 \}, \{ 2, 3 \}, \{ 3, 4 \}  \big\}     \\
    & \mbox{$ 
             \opnzer{}{1}{O_0}
             \opnone{}{2}{O_1}
             \trstwo{}{3}{T_2}
             \clstwo{}{4}{C_2}
             \clsone{}{5}{C_1}
            $}
    & \quad \;
    & \mbox{$ 
             \opnzer{}{1}{O_0}
             \updone{}{2}{U_1}
             \trstwo{}{3}{T_2}
             \clstwo{}{4}{C_2}
             \clsone{}{5}{C_1}
            $}
  \end{aligned}
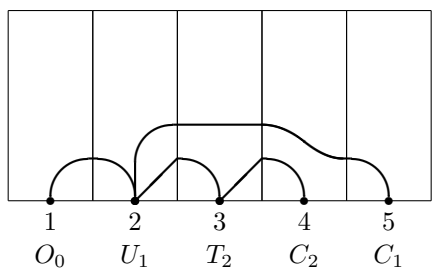
$$
}
\end{example}

\subsection{Enumeration and the weighted Motzkin paths}
\label{sec:5_2}

We give the generating function of the enumerating 
polynomials for some set partition statistics in non-crossing 
linked partitions by using the weighted Motzkin paths.

\smallskip 

If each step in a Motzkin path has a weight then it is called 
the weighted Motzkin path. 

\smallskip 

Consider the three sequences 
$$
  \{ \mu_0,     \mu_1,      \mu_2,     ... \}, \; 
  \{ \lambda_1, \lambda_2,  \lambda_3, ... \}, \, \mbox{ and } \,
  \{ \kappa_0,  \kappa_1,   \kappa_2,  ... \},
$$
which we use as the weights for the up, the down, and the transit 
steps in a Motzkin path, respectively.

Let $\mathbf{p} = (s_1, s_2, ... , s_n)$ be a Motzkin path of length $n$ 
where $s_j \in \{ u, d, t \}$. We make the associated list of weights 
$\mathbf{w}_{\mathbf{p}} = (w_1, w_2, ... , w_n)$ depending on both 
the type and the height of each step as follows:
$$
\begin{aligned}
 \; & \mbox{ if } s_j = u, & \mbox{ then } w_j &= \mu_{y_j} 
                           & \; \; &(j \ge 0), \\
 \; & \mbox{ if } s_j = d, & \mbox{ then } w_j &= \lambda_{y_j} 
                           & \; \; &(j \ge 1),\\
 \; & \mbox{ if } s_j = t, & \mbox{ then } w_j &= \kappa_{y_j} 
                           & \; \; &(j \ge 0).
\end{aligned}
$$
where $y_j$ is the height of the step $s_j$.

Given a weighted Motzkin path $\mathbf{p}$ of length $n$, 
the weight of the Motzkin path $wt(\mathbf{p})$ is defined by 
${\displaystyle 
  wt(\mathbf{p}) = \! \prod_{j=1}^n \! w_j
}$, 
the product of the weights in $\mathbf{w}_{\mathbf{p}}$.

\smallskip 

The following well-known formula on the generating function 
can be derived with help of the results in \cite{Fl80}.

\begin{theorem}
Let $\mathcal{M}_n$ be the set of all Motzkin paths of length $n$ 
and assume each of the paths is weighted in the manner above.
We write the sum of all the weights of the Motzkin paths in 
$\mathcal{M}_n$ by 
$$
  m_n = \! \! \sum_{\mathbf{p} \in \mathcal{M}_n} \! \!
                     wt(\mathbf{p}),
$$
which is sometimes called the $n$th moment. 
Then its generating function 
$$
   M(z) =\sum_{n = 0}^{\infty} m_n z^n 
$$
can be expanded into the continued fraction of the form 
$$
 M(z) = 
        \cfrac{1}{ 1 - \kappa_0 z - 
        \cfrac{ \mu_0 \lambda_1 z^2 }{ 1 - \kappa_1 z -
        \cfrac{ \mu_1 \lambda_2 z^2 }{ 1 - \kappa_2 z -
        \cfrac{ \mu_2 \lambda_3 z^2 }{ 1 - \kappa_3 z -
        \cfrac{ \mu_3 \lambda_4 z^2 }{ \ddots
        }}}}} .
$$
\end{theorem}

\smallskip 

This formula enables us to give the generating function of the 
enumerating polynomials for some set partition statistics 
in non-crossing linked partitions via the weighted Motzkin paths.

\begin{definition}{\rm
For a non-crossing linked partition $\pi \in \mathcal{NCL}(n)$, we shall 
introduce the following set partition statistics:
\begin{itemize}
\setlength{\leftskip}{7mm}
\vspace{-5pt}
\item[$dc(\pi)$:]
{the number of doubly covered elements by $\pi$}, 
\vspace{-5pt}
\item[$sc(\pi)$:]
{the number of singly covered minimal elements by $\pi$, 
but non-singleton}, 
\vspace{-5pt}
\item[$sg(\pi)$:]
{the number of singletons in $\pi$}.
\end{itemize}
}
\end{definition}

\begin{example}{\rm
For the partition $\pi$ in Example 5.3, each value of the statistics
above becomes 
   $dc(\pi) = 3$, $sc(\pi) = 2$, and $sg(\pi$) = 1.
}
\end{example}

\begin{remark}{\rm
We should note that the relation 
$$
   | \pi | = dc(\pi) + sc(\pi) + sg(\pi)
\eqno{(5)}
$$
holds because the minimal element of each block 
of $\pi \in \mathcal{NCL}(n)$ falls into one of the above three 
statistics, where $|\pi|$ stands for the number of blocks 
in $\pi$. 

Moreover, only the doubly covered elements are double-counted, 
thus we have the equality
$$
    \sum_{B \in \mbox{\normalsize $\pi$}} |B| = n + dc(\pi).
\eqno{(6)}
$$
}
\end{remark}

\smallskip 

We encode the joint statistics $(dc, sc, sg)$ in $\mathcal{NCL}(n)$ by 
$(\alpha, \beta, \gamma)$, and write the generating function of the 
enumerating polynomials in $\alpha$, $\beta$, and $\gamma$ as 
$$
   \Gamma (z; \alpha, \beta, \gamma) = \sum_{n=0}^{\infty} 
   \bigg( \! \!  \sum_{\; \; \lbpi \in \mathcal{NCL}(n)} \! \! 
    \alpha^{dc(\pi)} \, \beta^{sc(\pi)} \, \gamma^{sg(\pi)} \bigg) \, z^n.
$$

\smallskip 

Here we assign the weights to the cards and consider the weighted 
cards arrangements, which yield the weighted Motzkin paths.

\smallskip 

How to assign the weight to the cards is simple, that is, we will 
assign the weights $\alpha$, $\beta$, and $\gamma$ to the cards that 
correspond to the set partition statistics $dc$, $sc$, and $sg$, 
respectively. The cards that do not contribute to any statistics 
should be weighted $1$.  

\medskip 

\noindent
{\it The weight of the cards}
\begin{itemize}
\setlength{\leftskip}{7mm}
\item[($\alpha$)]
{The doubly covered elements correspond to the 
cards $T_i \, (i \ge 1)$ and $U_i \, (i \ge 1)$.
Thus these cards have the weight $\alpha$.}
\item[($\beta$)]
{The singly covered minimal elements of non-singleton 
correspond to the opening cards $O_i \, (i \ge 0)$. Thus 
these cards have the weight $\beta$.}
\item[($\gamma$)]
{The singletons, of course, are represented by the singleton 
cards $S_i \, (i \ge 0)$ which have the weight $\gamma$.}
\item[($1$)]{We assign the weight $1$ to the intermediate cards 
$I_i \, (i \ge 1)$ and the closing cards $C_i \, (i \ge 1)$, 
which do not correspond to any statistics.}
\end{itemize}

\smallskip 

Then we define the weight of an admissible arrangement as the 
product of the weights of the cards used in it.

\medskip 

For instance, the non-crossing linked partition in 
Example 5.3 can be represented by the admissible cards 
arrangement 
$$
  \big(O_0, U_1, S_2, C_2, O_1, C_2, T_1, I_1, T_1, C_1 \big).
$$
$$
             \opnzer{\beta }{ 1}{O_0}
             \updone{\alpha}{ 2}{U_1}
             \sgltwo{\gamma}{ 3}{S_2}
             \clstwo{1     }{ 4}{C_2}
             \opnone{\beta }{ 5}{O_1}
             \clstwo{1     }{ 6}{C_2}
             \trsone{\alpha}{ 7}{T_1}
             \imdone{1     }{ 8}{I_1}
             \trsone{\alpha}{ 9}{T_1}
             \clsone{1     }{10}{C_1}
$$
Then the weight of the arrangement is given by 
$$
  \beta  \cdot \alpha \cdot \gamma \cdot 1 \cdot \beta  \cdot 
  1      \cdot \alpha \cdot 1 \cdot \alpha \cdot 1 
  = \alpha^3 \beta^2 \gamma.
$$

\medskip 

\noindent
{\it The weights of the steps}

Now we can determine the weights of the steps in the weighted 
Motzkin paths for our statistics according to the rule of the 
arrangements for ${\mathcal{NCL}}(n)$ as follows:
\begin{itemize}
\setlength{\leftskip}{7mm}
\item[($u$)]
{For the up step at the ground level, only the 
opening card $O_0$ is available, thus $\mu_0 = \beta$. 
While for the up step at the height $i \ge 1$, the opening card 
$O_i \, (i \ge 1)$ and the card for the doubly covered elements of 
type II $U_i \, (i \ge 1)$ are available, thus 
$\mu_i = \alpha + \beta  \, (i \ge 1)$.}
\item[($t$)]
{For the transit step at the ground level, we can use only the singleton 
card $S_0$, thus $\kappa_0 = \gamma$. 
For the up step at the height $i \ge 1$, however, we can use 
the intermediate card $I_i \, (i \ge 1)$, the card for doubly 
covered elements of type I $T_i \, (i \ge 1)$, and 
the singleton card $S_i \, (i \ge 1)$, thus 
$\kappa_i = 1 + \alpha + \gamma  \, (i \ge 1)$.}
\item[($d$)]
{For the down step at the height $i \ge 1$, only the closing 
card $C_i \, (i \ge 1)$ is available, thus 
$\lambda_i = 1 \, (i \ge 1)$.}
\end{itemize}

Consequently, the weights of the steps for the weighted 
Motzkin path for our joint statistics $(\alpha, \beta, \gamma)$ 
should be given as follows:
$$
\left.
\begin{aligned}
  & \mbox{for up step}, & \; \; 
   \mu_{i} & = \mbox{$\displaystyle{ 
                   \Big\{
                    \begin{array}{ll}
                     \! \beta          \quad \; \; \; & (i = 0),  \\
                     \! \alpha + \beta \quad \; \; \; & (i \ge 1),
                    \end{array} 
                   }$}\\
  & \mbox{for transit step}, & \; \; 
\kappa_{i} & = \mbox{$\displaystyle{ 
                   \Big\{
                    \begin{array}{ll}
                     \! \gamma              \; & (i = 0),  \\
                     \! 1 + \alpha + \gamma \; & (i \ge 1),
                    \end{array} 
                    }$}\\
  & \mbox{for down step}, & \; \; 
     \lambda_{i} & = 1 \quad (i \ge 1).
\end{aligned}
\qquad \qquad 
\right\}
\eqno{(7)}
$$

\smallskip

\begin{example}{\rm
Each admissible arrangement for the Motzkin path 
$\mathbf{p} = (u,u,t,d,d)$ in Example 5.5 has the 
following weight:
$$
  \begin{aligned}
   (1) & \quad \mbox{wt} \big( O_0, O_1, I_2, C_2, C_1 \big) 
       & \quad \;
   (2) & \quad \mbox{wt} \big( O_0, U_1, I_2, C_2, C_1 \big) \\
       & \quad = \beta \cdot \beta  \cdot 1 \cdot 1 \cdot 1, 
       & \quad \;
       & \quad = \beta \cdot \alpha \cdot 1 \cdot 1 \cdot 1, \\
   (3) & \quad \mbox{wt} \big( O_0, O_1, S_2, C_2, C_1 \big) 
       & \quad \;
   (4) & \quad \mbox{wt} \big( O_0, U_1, S_2, C_2, C_1 \big) \\
       & \quad = \beta  \cdot  \beta  \cdot \gamma  \cdot 1,  \cdot 1 
       & \quad \;
       & \quad = \beta  \cdot  \alpha \cdot \gamma  \cdot 1,  \cdot 1 \\
   (5) & \quad \mbox{wt} \big( O_0, O_1, T_2, C_2, C_1 \big) 
       & \quad \;
   (6) & \quad \mbox{wt} \big( O_0, U_1, T_2, C_2, C_1 \big) \\
       & \quad = \beta  \cdot  \beta  \cdot \alpha  \cdot 1,  \cdot 1 
       & \quad \;
       & \quad = \beta  \cdot  \alpha \cdot \alpha  \cdot 1.  \cdot 1 
  \end{aligned}
$$
The sum of the above $6$ weights of the admissible arrangements is
$$
 \beta \cdot (\alpha + \beta) \cdot (1 + \alpha + \gamma) \cdot 
 1 \cdot 1 = \mu_0 \, \mu_1 \, \kappa_2 \, \lambda_2 \, \lambda_1,
$$
that is, the weight of the Motzkin path $\mathbf{p} = (u,u,t,d,d)$.
}
\end{example}

\smallskip

From Theorem 5.6, it follows that the generating function 
$\Gamma (z; \alpha, \beta, \gamma)$ can be obtained by the continued 
fraction of the form 
$$
   \Gamma (z; \alpha, \beta, \gamma) = 
        \mbox{$\displaystyle{
        \cfrac{1}{ 1 - \gamma \, z - 
        \cfrac{ \beta \, z^2 }{ 1 - (1 + \alpha + \gamma) \, z -
        \cfrac{ (\alpha + \beta) \, z^2 }{ 1 - (1 + \alpha + \gamma) \, z -
        \cfrac{ (\alpha + \beta) \, z^2 }{ \ddots
        }}}} 
        }$}\, ,
\eqno{(9)}
$$
which is rewritten as 
$$
 \Gamma (z; \alpha, \beta, \gamma) =
    \mbox{$\displaystyle{
    \cfrac{1}{ 1 - \gamma \, z - 
    \cfrac{\beta \, z^2}{h (z; \alpha, \beta, \gamma)}}
    }$} ,
$$
where the recursive part $h (z; \alpha, \beta, \gamma)$ satisfies 
the relation 
$$
   h (z; \alpha, \beta, \gamma)
 = 1 - (1 + \alpha + \gamma) \, z - \frac{ (\alpha + \beta) \, z^2}
   {h (z; \alpha, \beta, \gamma)}.
$$
Eliminating $h(z; \alpha, \beta, \gamma)$, we have the following formula:

\begin{theorem}
The generating function $\Gamma (z) = \Gamma (z; \alpha, \beta, \gamma)$ 
satisfies the quadratic equation,
$$
\mbox{$\displaystyle{
 \begin{aligned}
    &  \big( 1 + (\beta - \gamma) \, z \big) \, 
       \big( \alpha + (\beta - \alpha \, \gamma) \, z \big) \, \Gamma (z)^2 \\
    & \qquad 
             - \big\{ (2 \, \alpha + \beta) + 
                     \big( \, \beta \,(1 + \alpha + \gamma)  
                  - 2 \, (\alpha + \beta) \, \gamma \big) \, z 
               \big\} \,\Gamma (z)  + (\alpha + \beta) = 0,
 \end{aligned}
}$}
$$
and its closed form is solved as 
$$
 \Gamma (z; \alpha, \beta, \gamma)
    = \mbox{$\displaystyle{
        \frac{
         \Bigg\{ \mbox{$
          \displaystyle{
         \begin{aligned}
        & (2 \, \alpha + \beta) + \big( (1 + \alpha + \gamma) \, \beta 
         - 2 \, (\alpha + \beta) \, \gamma \big) \, z  \\
        &  \qquad \qquad 
           - \beta \, 
              \mbox{$\sqrt{ ( 1 - (1 + \alpha + \gamma) \, z )^2
              - 4 \, (\alpha + \beta) \, z^2}$}
         \end{aligned}}$} \; \Bigg\} }
       {\Big. 2 \, \big( 1 + (\beta - \gamma) \, z \big)  
              \big( \alpha + (\beta - \alpha \, \gamma) \, z \big) \Big. }
       }$} .
$$
\end{theorem}

\subsection{Combinatorial moment formulas}
\label{sec:5_2}

Related to the additive free convolution, the remarkable relation 
between the moments and the free cumulants (the coefficients of 
$R$-transform) was discovered by Speicher in \cite{Sp94}, which is 
known as the free moment-cumulant formula.

\smallskip

\begin{theorem}{\rm (\cite{Sp94})}
Let $\mu$ be a compactly supported probability measure on 
${\mathbb R}$ and denote its $R$-transform by 
$$
  R_\mu (z) = \sum_{k \ge 1} r_k z^{k-1}.
$$
Then the $n$th moment of $\mu$ can be given by the 
combinatorial formula 
$$
  m_n (\mu) = \sum_{\lbpi \in \mathcal{NC}(n)} \prod_{B \in \lbpi}
           r_{|B|},
$$
where $\mathcal{NC}(n)$ is the set of non-crossing partitions of 
$[n]$.
\end{theorem}

Like in the additive case, related to the multiplicative free 
convolution, the combinatorial formula of the moments by 
the coefficients of $T$-transform was found by Dykema 
in \cite{Dy07} (see also \cite{Ni10}).

\begin{theorem} {\rm (\cite{Dy07})}
Let $\mu$ be a compactly supported probability measure on 
$[0, \infty)$ with non-zero mean and denote its $T$-transform by 
$$
   T_{\mu} (z) = \sum_{k \ge 0} \alpha_k \, z^k,
$$
then the $n$th moment of $\mu$ can be given by the 
combinatorial formula
$$
  m_n (\mu)  = \!\!\!\! \sum_{\pi \in \mathcal{NCL} (n)} \!\!
              \alpha_0^{n - |\pi|}
              \prod_{B \in \pi} \alpha_{|B| - 1} 
             = \alpha_0^{n} \!\!\!\! \sum_{\pi \in \mathcal{NCL} (n)} 
              \prod_{B \in \pi} \Big( 
                 \frac{\alpha_{|B| - 1}}{\alpha_{0}} \Big),
$$
where $\mathcal{NCL} (n)$ is the set of non-crossing linked 
partitions of $[n]$.
\end{theorem}

\begin{remark}{\rm
Although $\mathcal{NCL} (n)$ can not exactly make a lattice, 
the moments and the coefficients of $T$-transform determine 
each other in the same manner as in the additive free 
moment-cumulant formula.
}
\end{remark}

\smallskip

At the end of this section we introduce the distinguished 
operator on the full space $\mathcal{T}(\mathcal{H})$ whose moments 
are closely related to the non-crossing linked partitions.

\begin{proposition}
Let $\mathcal{H}$ be a Hilbert space and $\mathcal{T}(\mathcal{H})$ 
be a full Fock space of $\mathcal{H}$.
For a unit vector $\xi \in \mathcal{H}$, let $\ell = \ell(\xi)$ 
and $\ell^{*} = \ell(\xi)^{*}$ be the left creation and the left 
annihilation operators on a full Fock space 
$\mathcal{T}(\mathcal{H})$, respectively. 
We set the operator $X$ in 
$\mathcal{B} \big( \mathcal{T}(\mathcal{H})\big)$ as
$$
 X = \gamma \, {\bm 1} + 
      \beta \, \ell + \ell^* + (1 + \alpha) \, \ell \ell^* 
      + \alpha \, \ell^2 \ell^*.
$$
Then the $n$th moment of the operator $X$ with respect to 
the vacuum expectation $\varphi$ is given by
$$
  \varphi \big( X^n \big) = \! \! \! \! 
   \sum_{\; \; \pi \in \mathcal{NCL}(n)} \! \! 
     \alpha^{dc(\pi)} \, \beta^{sc(\pi)} \, \gamma^{sg(\pi)}.
$$
Hence the moment generating function of the random variable $X$ is 
given by $\Gamma (z; \alpha, \beta, \gamma)$ in Theorem 5.11.
\end{proposition}

\begin{proof}
We decompose the operator $X$ and set $u$, $d$, $t$ as follows: 
$$
 \begin{aligned}
 X & = \gamma \, {\bm 1} + 
      \beta \, \ell + \ell^* + (1 + \alpha) \, \ell \ell^* 
      + \alpha \, \ell^2 \ell^* \\
   & = \underbrace{\alpha \, \ell^2 \ell^* 
        + \beta \, \ell}_{\mbox{\normalsize{$u$}}}  
    + \underbrace{\ell^*}_{\mbox{\normalsize{$d$}}} 
    + \underbrace{\gamma \, {\bm 1} 
        + (1 + \alpha) \, \ell \ell^* }_{\mbox{\normalsize{$t$}}} .
 \end{aligned}
$$
Then the operators $u$, $d$, and $t$ act on the elementary vectors 
$\xi^{\otimes n} \in \mathcal{T}(\mathcal{H})$ $(n \ge 0)$ that 
$$
 \begin{aligned}
  u \, \xi^{\otimes n} & = 
   \begin{cases} 
     \beta \, \xi^{\otimes 1}                  & \mbox{ if }  n = 0, \\
     (\alpha + \beta) \, \xi^{\otimes (n + 1)} & \mbox{ if }  n \ge 1, 
   \end{cases} 
   \qquad 
  d \, \xi^{\otimes n}   = 
   \begin{cases} 
      0                     & \mbox{ if }  n = 0, \\
      \xi^{\otimes (n - 1)} & \mbox{ if }  n \ge 1.
   \end{cases} \\
  t \, \xi^{\otimes n} & = 
   \begin{cases} 
      \gamma \Omega                            & \mbox{ if }  n = 0, \\
      (1 + \alpha + \gamma) \, \xi^{\otimes n} & \mbox{ if }  n \ge 1, 
   \end{cases}
 \end{aligned}
$$
where $\xi^{\otimes 0}$ conventionally means the vacuum vector $\Omega$.
By using the weights for the weighted Motzkin paths in (7), 
we can simply write 
$$
 \begin{aligned}
  u \, \xi^{\otimes n} & = \mu_n \, \xi^{\otimes (n + 1)} \; (n \ge 0), & 
  \qquad 
  d \, \xi^{\otimes n} & = \lambda_n \, \xi^{\otimes (n - 1)} \; (n \ge 1), \\
  t \, \xi^{\otimes n} & = \kappa_n \, \xi^{\otimes n} \; (n \ge 0). & 
 \end{aligned}
$$

Now we expand $X^n = (u + d + t)^n$ into $3^n$ non-commutative 
monomials in $s$, $d$, and $t$ as 
$$
  \sum_{s_i \in \{ u, d, t \}
        (i=1,\ldots, n)}  
   s_n s_{n-1} \cdots s_1, 
$$
then only the monomials of the Motzkin path type could 
contribute to the vacuum expectation 
$\varphi(\, \cdot \,) = \langle \, \cdot \, \Omega | \Omega \rangle$.
Indeed, it follows that 
$s_n s_{n-1} \cdots s_1  \, \Omega \in \mathbb{C} \, \Omega$ if and 
only if $s_1 s_2 \cdots s_n$ makes a Motzkin path of length $n$.
In this case, 
$$
 \varphi(s_n s_{n-1} \cdots s_1) 
 = \langle \, s_n s_{n-1} \cdots s_1 \Omega | \Omega \rangle
 = wt(\mathbf{p}), 
$$
where $wt(\mathbf{p})$ is the weight of the weighted Motzkin 
path $\mathbf{p}$. Hence it follows that 
$$
 \varphi \big( X^n \big) 
 = \sum_{\bm{p} \in \mathcal{M}} wt(\mathbf{p}), 
$$
where the weight sequences for the weighted Motzkin paths are given 
as in (7).

\smallskip 

Here we should note that the following correspondence 
between the operators $u$, $d$, $t$ and the cards for 
$\mathcal{NCL}(n)$ can be found by decomposition of $u$ and $t$:
$$
 u = \underbrace{\big. \alpha \, \ell^2 \ell^*}_{\mbox{\small{$U$ card}}} 
   + \underbrace{\big. \beta \, \ell}_{\mbox{\small{$O$ card}}},
\quad 
 d = \underbrace{\big. \ell^*}_{\mbox{\small{$C$ card}}}, 
\quad 
 t = \underbrace{\big. \gamma \, {\bm 1}}_{\mbox{\small{$S$ card}}}
   + \underbrace{\big. \ell \ell^* }_{\mbox{\small{$I$ card}}}
   + \underbrace{\big. \alpha \, \ell \ell^* }_{\mbox{\small{$T$ card}}},
$$
where each coefficient corresponds to the weight of the card. We can also 
find that only the admissible arrangements of $n$ cards could contribute 
to the vacuum expectation of $X^n$. 
\end{proof}

\begin{remark}{\rm
The operator 
$$
  {\widetilde X} = 
   \underbrace{\gamma {\bm 1} + 
     \sqrt{\beta} \big( \ell + \ell^* \big) 
     + \ell \ell^* }_{\footnotesize \mbox{shifted free Poisson}}
    + \alpha ({\bm 1} + \ell)  \ell \ell^*
$$
has the same distribution as of the operator $X$, and 
the shifted free Poisson part is contained in ${\widetilde X}$ 
because 
$\displaystyle{ 
  \beta {\bm 1} + 
     \sqrt{\beta} \big( \ell + \ell^* \big) + \ell \ell^* 
}$
corresponds to the free Poisson random variable of 
parameter $\beta$. 
The operators in such a form of $X$ or ${\widetilde X}$ were 
also investigated by Bo{\.z}ejko and Lytvynov in \cite{BoLy09}.
}
\end{remark}

%
%
\section{The moments of the free beta prime distribution}
\label{sec:6}
%
%

In this section, we investigate the moments of the free beta 
prime distribution $\fbepr$ and see that $\fbepr$ is in the free Meixner 
family. Moreover we determine the free Meixner class to which 
the free beta prime distribution should be classified.

\smallskip

Since we know that the $S$-transform of $\fbepr$ is given by 
(2) in the proof of Proposition 3.2, the $T$-transform of $\fbepr$ 
is obtained as 
$\displaystyle{T_{\fbepr} (z) = \frac{z + a}{b -1 - z}}$, 
which has the following expansion:
$$
  T_{\fbepr} (z) = \frac{a}{b - 1} + (a + b - 1) 
       \sum_{k = 1}^{\infty} \frac{z^k}{(b - 1)^{k+1}} \\
  = \sum_{k=0}^{\infty} \alpha_k z^k.
$$
Thus the coefficients of the $T$-transform become 
$$
   \alpha_0 = \frac{a}{b - 1}, \; \;
   \alpha_k = \frac{a + b - 1}{(b - 1)^{k+1}} \; \;
   (k \ge 1).
$$
In order to simplify the expression, we put 
$$
   s = \frac{a}{b - 1}, \; \;
   t = \frac{a + b - 1}{b - 1} \;\mbox{ and } \;
   u = \frac{1}{b - 1}, 
$$
then we can write 
$\alpha_0 = s, \; \; \alpha_k = t \, u^k \; \; (k \ge 1)$
in short.

\medskip 

Applying the moment formula in Theorem 5.13, the $n$th moment 
of $\fbepr$ is given in a combinatorial form by 
$$
 \begin{aligned}
  m_n\big( \fbepr \big) 
      = & \sum_{\pi \in {\mathcal{NCL}(n)}}
           s^{n - |\pi|}  \! 
           \prod_{\substack{\lsblk{B} \in \pi \\ |\lsblk{B}| \ne 1}}
           t \, u^{|\lsblk{B}|-1} \! 
           \prod_{\substack{\lsblk{B} \in \pi \\ |\lsblk{B}| = 1}} 
            s 
          \\
      = & \sum_{\pi \in {\mathcal{NCL}(n)}}
           s^{n - |\pi|}  \! 
           \prod_{\lsblk{B} \in \pi}
           t \, u^{|\lsblk{B}|-1}  \! 
           \prod_{\substack{\lsblk{B} \in \pi \\ |\lsblk{B}| = 1}} \! 
           \Big( \frac{s}{\, t \,} \Big) .
  \end{aligned}
\eqno{(8)}
$$
Using the relations (5) and (6) among the set partition 
statistics $sg$, $sc$, and $dc$ in Remark 5.9, 
the first product in the most right hand side of (8) can be 
reformulated as 
$$
 \begin{aligned}
   \prod_{\lsblk{B} \in \pi} t u^{|\lsblk{B}|-1} 
  & = \prod_{\lsblk{B} \in \pi}  \Big( \frac{t}{\, u \, } \Big) u^{|\lsblk{B}|}
    = \Big( \frac{t}{\, u \, } \Big)^{|\pi|} 
      u^{ \sum_{B \in \pi} |B|} \\
  & = \Big( \frac{t}{\, u \, } \Big)^{dc(\pi) + sc(\pi) + sg(\pi)} 
       u^{n + dc(\pi)} 
    = u^n \, t^{dc(\pi)} 
      \Big( \frac{t}{\, u \, } \Big)^{sc(\pi)}
      \Big( \frac{t}{\, u \, } \Big)^{sg(\pi)}.
 \end{aligned}
$$
It is clear that the second product in (8) becomes 
$$
  \prod_{\substack{\lsblk{B} \in \pi \\ |\lsblk{B}| = 1}} \! 
        \Big( \frac{s}{\, t \,} \Big)
    = \Big( \frac{s}{\, t \, } \Big)^{sg(\pi)}, 
$$
and that 
$$
   s^{n - |\pi|} = s^{n} \Big( \frac{1}{\, s \, } \Big)^{dc(\pi)}
      \Big( \frac{1}{\, s \, } \Big)^{sc(\pi)}
      \Big( \frac{1}{\, s \, } \Big)^{sg(\pi)}.
$$
Consequently, we obtain the following combinatorial formula of 
the $n$th moment of $\fbepr$.

\smallskip

\begin{theorem}
The $n$th moment of the free beta prime distribution $\fbepr$
is given by 
$$
  m_n\big( \fbepr \big) 
       = (s u)^n \! \! \! \! 
         \sum_{\pi \in {\mathcal{NCL}(n)}}  \! \! 
          \Big( \frac{t}{\, s \,} \Big)^{\! dc(\pi)}
          \Big( \frac{t}{s u}     \Big)^{\! sc(\pi)}
          \Big( \frac{1}{\, u \,} \Big)^{\! sg(\pi)}, 
$$
where 
$\displaystyle{s = \frac{a}{b - 1}}$, 
$\displaystyle{t = \frac{a + b - 1}{b - 1}}$, and 
$\displaystyle{u = \frac{1}{b - 1}}$, 
equivalently, it is given by 
$$
  m_n\big( \fbepr \big) 
       = \bigg( \frac{a}{(b-1)^2} \bigg)^{\!n} \! \! \! 
       \sum_{\pi \in {\mathcal{NCL}(n)}}  \! \! 
         \Big( \frac{a + b - 1}{a} \Big)^{\! dc(\pi)+ sc(\pi)}
         \big( b - 1 \big)^{\! sc(\pi)+ sg(\pi)}.
$$
A model of the $\fbepr$-distributed random variable $X(a,b)$ 
on a full Fock space $\mathcal{T}(\mathcal{H})$ can be given by 
the form 
$$
 \begin{aligned}
 X(a, b) = &  \sfrac{a}{(b - 1)^2} 
    \Big\{(b - 1) \, {\bm 1} + 
      \sfrac{(a + b - 1)(b - 1)}{a} \, \ell + \ell^* + \ell \ell^* 
     + \sfrac{a + b - 1}{a} ({\bm 1} + \ell)  \ell \ell^* \Big\} \\
  = &  \sfrac{a}{b - 1} \, {\bm 1}
     + \sfrac{a + b - 1}{b - 1}  \, \ell 
     + \sfrac{a}{(b - 1)^2}  \, \ell^* 
     + \sfrac{2 a + b - 1}{(b - 1)^2} \, \ell \ell^* 
     + \sfrac{a + b - 1}{(b - 1)^2} \, \ell^2 \ell^*, 
 \end{aligned}
$$
where $\ell = \ell(\xi)$ and $\ell^{*} = \ell(\xi)^{*}$ are 
the left creation and the left annihilation operators for 
a unit vector $\xi \in \mathcal{H}$, respectively.
\end{theorem}

The model of the random variable $X(a,b)$ is an immediate consequence 
of Proposition 5.15.

\medskip

A family of the free Meixner distributions contains many important 
laws in free probability, which has been investigated in many 
literatures, for instance, \cite{An03},  \cite{BoBr06},  \cite{By10},
\cite{Ej12}, and \cite{SY01}.

In particular, Bo{\. z}ejko and Bryc in \cite{BoBr06} showed that the 
family of the standard (mean $0$ and variance $1$) free Meixner 
distributions is parameterized by $\theta$ and $\tau$ as 
\vspace{-5pt}
$$ 
  \big\{ \mu(\theta, \tau) \, : \,  
        \theta \in {\mathbb R}, \tau \ge -1 \big\}
\vspace{-5pt}
$$ 
with the Cauchy transform 
$$
 G_{\mu(\theta, \tau)} (z) = 
     \frac{(1 + 2 \tau) z + \theta - \sqrt{(z - \theta)^2 - 4 (1 + \tau)}}
          {2 (\tau z^2 + \theta z  +1)}.
$$
They also showed that the free Meixner distributions can be 
classified into six classes, which was inspired by the fact that 
the classical Meixner distributions satisfy Laha-Lukacs properties 
with similar parameters.

\smallskip

Here we will see that the free beta prime distribution $\fbepr$ is in 
a family of the free Meixner distributions and determine its class 
in the free Meixner family.

Let $X(a,b)$ be a free beta prime $\fbepr$ distributed random variable.
By the moment formula, it is easy to find that $X(a,b)$ has mean 
$m = \displaystyle{\sfrac{a}{b - 1}}$ and 
variance $v = \displaystyle{\sfrac{a (a + b - 1)}{(b - 1)^3}}$.

We shall standardize $X(a,b)$ as 
$\widetilde{X}(a,b) = \dfrac{X(a,b) - m}{\sqrt{v}}$,
since we know the Cauchy transform $G_{X(a,b)}$ as in (1) 
in Proposition 3.2, the Cauchy transform of $\widetilde{X}(a,b)$ 
can be obtained by direct calculation as follows:
$$
 \begin{aligned}
   G_{\widetilde{X}(a,b)} (z) & =  
   \sqrt{v} \, G_{X(a,b)}  \big( \sqrt{v} z + m \big) \\
 & =  \frac{ \; 
      \left(
      \begin{aligned}
        & \dfrac{1 + b}{1 - b} \, z  
         + \dfrac{2 a + b - 1}{\sqrt{ a (a + b - 1)(b - 1)}} \\
        & \qquad \; \;  - \sqrt{ \, z^2 
          - 2 \, \frac{2 a + b - 1}{\sqrt{ a (a + b - 1)(b - 1)}} \, z
      + \frac{b - 1}{a (a + b - 1)}  - 4} 
      \end{aligned} \right) \;
     }
     {2 \, \bigg( 
         \dfrac{1}{b - 1} \, z^2 
        + \dfrac{2 a + b - 1}{\sqrt{ a (a + b - 1)(b - 1)}} \, z + 1 \bigg)}.
 \end{aligned}
$$
If we put 
$$
 \theta = \frac{2 a + b - 1}{\sqrt{ a (a + b - 1)(b - 1)}} \;
 \mbox{ and } \; 
 \tau = \frac{1}{b - 1}
$$
then it follows that 
$$
 \begin{aligned}
   \big( z - \theta \big)^2 - 4 (1 + \tau) 
  & =  z^2 - 2 \, \frac{2 a + b - 1}{\sqrt{ a (a + b - 1)(b - 1)}} \, z
      + \frac{b - 1}{a (a + b - 1)}  - 4, \\
   1 + 2 \tau & = \frac{1 + b}{1 - b},
 \end{aligned}
$$
which means that the distribution of $\widetilde{X}(a,b)$ is in a family 
of the free Meixner distribution with the above parameters $\theta$ and
$\tau$.

\smallskip 

According to the classification table by \cite{BoBr06}, the distribution 
of the random variable $\widetilde{X}(a,b)$ is classified into the free 
negative binomial distributions because the classification 
parameters satisfy the inequality 
$$
   \theta^2 - 4 \tau = \frac{b - 1}{a (a + b -1)}  > 0.
$$

Hence we can conclude the following proposition:

\begin{proposition}
The free beta prime distribution $\fbepr$ is in the class of 
the free negative binomial distributions and, hence, it is freely 
infinitely divisible.
\end{proposition}







\end{document}

%% file: cardsfig.tex
\newcommand{\opnzer}[3]{
\setlength{\unitlength}{1.6pt}
\begin{picture}(20,58)(0,-11)
   \thinlines
   \multiput(  0, 0)(20,0){2}{\line(0,1){45}}
   \multiput(  0, 0)(0,45){2}{\line(1,0){20}}
   \put(  0, 36){\makebox(10, 7){{\small $#1$}}}
   \put(  0, -9){\makebox(20, 8){{\small $#2$}}}
   \put(  0,-16){\makebox(20, 6){{\small $#3$}}}
   \thicklines
   \put( 19, 1){\oval(18,18)[tl]}
   \put( 19,10){\line(1,0){1}}
   \put( 10, 0){\circle*{2}}
\end{picture}
}
\newcommand{\opnone}[3]{
\setlength{\unitlength}{1.6pt}
\begin{picture}(20,58)(0,-11)
   \thinlines
   \multiput(  0, 0)(20,0){2}{\line(0,1){45}}
   \multiput(  0, 0)(0,45){2}{\line(1,0){20}}
   \put(  0, 36){\makebox(10, 7){{\small $#1$}}}
   \put(  0, -9){\makebox(20, 8){{\small $#2$}}}
   \put(  0,-16){\makebox(20, 6){{\small $#3$}}}
   \thicklines
   \put( 19, 1){\oval(18,18)[tl]}
   \put( 19,10){\line(1,0){1}}
   \put( 10, 0){\circle*{2}}
   \qbezier( 0,10)( 5,10)(10,14)
   \qbezier(19,18)(15,18)(10,14)
   \put(19,18){\line(1,0){1}}
\end{picture}
}
\newcommand{\opntwo}[3]{
\setlength{\unitlength}{1.6pt}
\begin{picture}(20,58)(0,-11)
   \thinlines
   \multiput(  0, 0)(20,0){2}{\line(0,1){45}}
   \multiput(  0, 0)(0,45){2}{\line(1,0){20}}
   \put(  0, 36){\makebox(10, 7){{\small $#1$}}}
   \put(  0, -9){\makebox(20, 8){{\small $#2$}}}
   \put(  0,-16){\makebox(20, 6){{\small $#3$}}}
   \thicklines
   \put( 19, 1){\oval(18,18)[tl]}
   \put( 19,10){\line(1,0){1}}
   \put( 10, 0){\circle*{2}}
   \qbezier( 0,10)( 5,10)(10,14)
   \qbezier(19,18)(15,18)(10,14)
   \put(19,18){\line(1,0){1}}
   \qbezier( 0,18)( 5,18)(10,22)
   \qbezier(19,26)(15,26)(10,22)
   \put(19,26){\line(1,0){1}}
\end{picture}
}
\newcommand{\opnnth}[3]{
\setlength{\unitlength}{1.6pt}
\begin{picture}(20,58)(0,-11)
   \thinlines
   \multiput(  0, 0)(20,0){2}{\line(0,1){45}}
   \multiput(  0, 0)(0,45){2}{\line(1,0){20}}
   \put(  0, 36){\makebox(10, 7){{\small $#1$}}}
   \put(  0, -9){\makebox(20, 8){{\small $#2$}}}
   \put(  0,-16){\makebox(20, 6){{\small $#3$}}}
   \thicklines
   \put( 19, 1){\oval(18,18)[tl]}
   \put( 19,10){\line(1,0){1}}
   \put( 10, 0){\circle*{2}}
   \qbezier( 0,10)( 4,10)(10,14)
   \qbezier(19,18)(16,18)(10,14)
   \put(19,18){\line(1,0){1}}
   \qbezier( 0,18)( 5,18)(10,22)
   \qbezier(19,26)(15,26)(10,22)
   \put(19,26){\line(1,0){1}}
   \put(10,32){\makebox(0,0){$.$}}
   \put(10,29){\makebox(0,0){$.$}}
   \put(10,26){\makebox(0,0){$.$}}
   \qbezier( 0,34)( 4,34)(10,38)
   \qbezier(19,42)(16,42)(10,38)
   \put(19,42){\line(1,0){1}}
\end{picture}
}
\newcommand{\updone}[3]{
\setlength{\unitlength}{1.6pt}
\begin{picture}(20,58)(0,-11)
   \thinlines
   \multiput(  0, 0)(20,0){2}{\line(0,1){45}}
   \multiput(  0, 0)(0,45){2}{\line(1,0){20}}
   \put(  0, 36){\makebox(10, 7){{\small $#1$}}}
   \put(  0, -9){\makebox(20, 8){{\small $#2$}}}
   \put(  0,-16){\makebox(20, 6){{\small $#3$}}}
   \thicklines
   \put(  1, 1){\oval(18,18)[tr]}
   \put(  0,10){\line(1,0){1}}
   \put( 10, 0){\circle*{2}}
   \put( 19, 9){\oval(18,18)[tl]}
   \put( 19,18){\line(1,0){1}}
   \put( 10, 1){\line(0,1){8}}
   \put( 10, 0){\circle*{2}}
   \put( 10, 0){\line(1,1){10}}
   \put( 10, 0){\circle*{2}}
\end{picture}
}
\newcommand{\updtwo}[3]{
\setlength{\unitlength}{1.6pt}
\begin{picture}(20,58)(0,-11)
   \thinlines
   \multiput(  0, 0)(20,0){2}{\line(0,1){45}}
   \multiput(  0, 0)(0,45){2}{\line(1,0){20}}
   \put(  0, 36){\makebox(10, 7){{\small $#1$}}}
   \put(  0, -9){\makebox(20, 8){{\small $#2$}}}
   \put(  0,-16){\makebox(20, 6){{\small $#3$}}}
   \thicklines
   \put(  1, 1){\oval(18,18)[tr]}
   \put(  0,10){\line(1,0){1}}
   \put( 10, 0){\circle*{2}}
   \put( 19, 9){\oval(18,18)[tl]}
   \put( 19,18){\line(1,0){1}}
   \put( 10, 1){\line(0,1){8}}
   \put( 10, 0){\circle*{2}}
   \put( 10, 0){\line(1,1){10}}
   \put( 10, 0){\circle*{2}}
   \qbezier( 0,18)( 5,18)(10,22)
   \qbezier(19,26)(15,26)(10,22)
   \put(19,26){\line(1,0){1}}
\end{picture}
}
\newcommand{\updthr}[3]{
\setlength{\unitlength}{1.6pt}
\begin{picture}(20,58)(0,-11)
   \thinlines
   \multiput(  0, 0)(20,0){2}{\line(0,1){45}}
   \multiput(  0, 0)(0,45){2}{\line(1,0){20}}
   \put(  0, 36){\makebox(10, 7){{\small $#1$}}}
   \put(  0, -9){\makebox(20, 8){{\small $#2$}}}
   \put(  0,-16){\makebox(20, 6){{\small $#3$}}}
   \thicklines
   \put(  1, 1){\oval(18,18)[tr]}
   \put(  0,10){\line(1,0){1}}
   \put( 10, 0){\circle*{2}}
   \put( 19, 9){\oval(18,18)[tl]}
   \put( 19,18){\line(1,0){1}}
   \put( 10, 1){\line(0,1){8}}
   \put( 10, 0){\circle*{2}}
   \put( 10, 0){\line(1,1){10}}
   \put( 10, 0){\circle*{2}}
   \qbezier( 0,18)( 5,18)(10,22)
   \qbezier(19,26)(15,26)(10,22)
   \put(19,26){\line(1,0){1}}
   \qbezier( 0,26)( 5,26)(10,30)
   \qbezier(19,34)(15,34)(10,30)
   \put(19,34){\line(1,0){1}}
\end{picture}
}
\newcommand{\updnth}[3]{
\setlength{\unitlength}{1.6pt}
\begin{picture}(20,58)(0,-11)
   \thinlines
   \multiput(  0, 0)(20,0){2}{\line(0,1){45}}
   \multiput(  0, 0)(0,45){2}{\line(1,0){20}}
   \put(  0, 36){\makebox(10, 7){{\small $#1$}}}
   \put(  0, -9){\makebox(20, 8){{\small $#2$}}}
   \put(  0,-16){\makebox(20, 6){{\small $#3$}}}
   \thicklines
   \put(  1, 1){\oval(18,18)[tr]}
   \put(  0,10){\line(1,0){1}}
   \put( 10, 0){\circle*{2}}
   \put( 19, 9){\oval(18,18)[tl]}
   \put( 19,18){\line(1,0){1}}
   \put( 10, 1){\line(0,1){8}}
   \put( 10, 0){\circle*{2}}
   \put( 10, 0){\line(1,1){10}}
   \put( 10, 0){\circle*{2}}
   \qbezier( 0,18)( 5,18)(10,22)
   \qbezier(19,26)(15,26)(10,22)
   \put(19,26){\line(1,0){1}}
   \put(10,32){\makebox(0,0){$.$}}
   \put(10,29){\makebox(0,0){$.$}}
   \put(10,26){\makebox(0,0){$.$}}
   \qbezier( 0,34)( 4,34)(10,38)
   \qbezier(19,42)(16,42)(10,38)
   \put(19,42){\line(1,0){1}}
\end{picture}
}
\newcommand{\clsone}[3]{
\setlength{\unitlength}{1.6pt}
\begin{picture}(20,58)(0,-11)
   \thinlines
   \multiput(  0, 0)(20,0){2}{\line(0,1){45}}
   \multiput(  0, 0)(0,45){2}{\line(1,0){20}}
   \put(  0, 36){\makebox(10, 7){{\small $#1$}}}
   \put(  0, -9){\makebox(20, 8){{\small $#2$}}}
   \put(  0,-16){\makebox(20, 6){{\small $#3$}}}
   \thicklines
   \put(  1, 1){\oval(18,18)[tr]}
   \put(  0,10){\line(1,0){1}}
   \put( 10, 0){\circle*{2}}
\end{picture}
}
\newcommand{\clstwo}[3]{
\setlength{\unitlength}{1.6pt}
\begin{picture}(20,58)(0,-11)
   \thinlines
   \multiput(  0, 0)(20,0){2}{\line(0,1){45}}
   \multiput(  0, 0)(0,45){2}{\line(1,0){20}}
   \put(  0, 36){\makebox(10, 7){{\small $#1$}}}
   \put(  0, -9){\makebox(20, 8){{\small $#2$}}}
   \put(  0,-16){\makebox(20, 6){{\small $#3$}}}
   \thicklines
   \put(  1, 1){\oval(18,18)[tr]}
   \put(  0,10){\line(1,0){1}}
   \put( 10, 0){\circle*{2}}
   \qbezier( 0,18)( 5,18)(10,14)
   \qbezier(19,10)(15,10)(10,14)
   \put(19,10){\line(1,0){1}}
\end{picture}
}
\newcommand{\clsthr}[3]{
\setlength{\unitlength}{1.6pt}
\begin{picture}(20,58)(0,-11)
   \thinlines
   \multiput(  0, 0)(20,0){2}{\line(0,1){45}}
   \multiput(  0, 0)(0,45){2}{\line(1,0){20}}
   \put(  0, 36){\makebox(10, 7){{\small $#1$}}}
   \put(  0, -9){\makebox(20, 8){{\small $#2$}}}
   \put(  0,-16){\makebox(20, 6){{\small $#3$}}}
   \thicklines
   \put(  1, 1){\oval(18,18)[tr]}
   \put(  0,10){\line(1,0){1}}
   \put( 10, 0){\circle*{2}}
   \qbezier( 0,18)( 5,18)(10,14)
   \qbezier(19,10)(15,10)(10,14)
   \put(19,10){\line(1,0){1}}
   \qbezier( 0,26)( 5,26)(10,22)
   \qbezier(19,18)(15,18)(10,22)
   \put(19,18){\line(1,0){1}}
\end{picture}
}
\newcommand{\clsnth}[3]{
\setlength{\unitlength}{1.6pt}
\begin{picture}(20,58)(0,-11)
   \thinlines
   \multiput(  0, 0)(20,0){2}{\line(0,1){45}}
   \multiput(  0, 0)(0,45){2}{\line(1,0){20}}
   \put(  0, 36){\makebox(10, 7){{\small $#1$}}}
   \put(  0, -9){\makebox(20, 8){{\small $#2$}}}
   \put(  0,-16){\makebox(20, 6){{\small $#3$}}}
   \thicklines
   \put(  1, 1){\oval(18,18)[tr]}
   \put(  0,10){\line(1,0){1}}
   \put( 10, 0){\circle*{2}}
   \qbezier( 0,18)( 4,18)(10,14)
   \qbezier(19,10)(16,10)(10,14)
   \put(19,10){\line(1,0){1}}
   \put(10,24){\makebox(0,0){$ \vdots $}}
   \qbezier( 0,34)( 4,34)(10,30)
   \qbezier(19,26)(16,26)(10,30)
   \put(19,26){\line(1,0){1}}
\end{picture}
}
\newcommand{\imdone}[3]{
\setlength{\unitlength}{1.6pt}
\begin{picture}(20,58)(0,-11)
   \thinlines
   \multiput(  0, 0)(20,0){2}{\line(0,1){45}}
   \multiput(  0, 0)(0,45){2}{\line(1,0){20}}
   \put(  0, 36){\makebox(10, 7){{\small $#1$}}}
   \put(  0, -9){\makebox(20, 8){{\small $#2$}}}
   \put(  0,-16){\makebox(20, 6){{\small $#3$}}}
   \thicklines
   \put(  1, 1){\oval(18,18)[tr]}
   \put(  0,10){\line(1,0){1}}
   \put( 10, 0){\circle*{2}}
   \put( 19, 1){\oval(18,18)[tl]}
   \put( 19,10){\line(1,0){1}}
   \put( 10, 0){\circle*{2}}
\end{picture}
}
\newcommand{\imdtwo}[3]{
\setlength{\unitlength}{1.6pt}
\begin{picture}(20,58)(0,-11)
   \thinlines
   \multiput(  0, 0)(20,0){2}{\line(0,1){45}}
   \multiput(  0, 0)(0,45){2}{\line(1,0){20}}
   \put(  0, 36){\makebox(10, 7){{\small $#1$}}}
   \put(  0, -9){\makebox(20, 8){{\small $#2$}}}
   \put(  0,-16){\makebox(20, 6){{\small $#3$}}}
   \thicklines
   \put(  1, 1){\oval(18,18)[tr]}
   \put(  0,10){\line(1,0){1}}
   \put( 10, 0){\circle*{2}}
   \put( 19, 1){\oval(18,18)[tl]}
   \put( 19,10){\line(1,0){1}}
   \put( 10, 0){\circle*{2}}
   \put( 0, 18){\line(1,0){20}}
\end{picture}
}
\newcommand{\imdthr}[3]{
\setlength{\unitlength}{1.6pt}
\begin{picture}(20,58)(0,-11)
   \thinlines
   \multiput(  0, 0)(20,0){2}{\line(0,1){45}}
   \multiput(  0, 0)(0,45){2}{\line(1,0){20}}
   \put(  0, 36){\makebox(10, 7){{\small $#1$}}}
   \put(  0, -9){\makebox(20, 8){{\small $#2$}}}
   \put(  0,-16){\makebox(20, 6){{\small $#3$}}}
   \thicklines
   \put(  1, 1){\oval(18,18)[tr]}
   \put(  0,10){\line(1,0){1}}
   \put( 10, 0){\circle*{2}}
   \put( 19, 1){\oval(18,18)[tl]}
   \put( 19,10){\line(1,0){1}}
   \put( 10, 0){\circle*{2}}
   \put( 0, 18){\line(1,0){20}}
   \put( 0, 26){\line(1,0){20}}
\end{picture}
}
\newcommand{\imdnth}[3]{
\setlength{\unitlength}{1.6pt}
\begin{picture}(20,58)(0,-11)
   \thinlines
   \multiput(  0, 0)(20,0){2}{\line(0,1){45}}
   \multiput(  0, 0)(0,45){2}{\line(1,0){20}}
   \put(  0, 36){\makebox(10, 7){{\small $#1$}}}
   \put(  0, -9){\makebox(20, 8){{\small $#2$}}}
   \put(  0,-16){\makebox(20, 6){{\small $#3$}}}
   \thicklines
   \put( 19, 1){\oval(18,18)[tl]}
   \put( 19,10){\line(1,0){1}}
   \put( 10, 0){\circle*{2}}
   \put(  1, 1){\oval(18,18)[tr]}
   \put(  0,10){\line(1,0){1}}
   \put( 10, 0){\circle*{2}}
   \put( 0, 18){\line(1,0){20}}
   \put(10,27){\makebox(0,0){$ \vdots $}}
   \put( 0, 34){\line(1,0){20}}
\end{picture}
}
\newcommand{\trsone}[3]{
\setlength{\unitlength}{1.6pt}
\begin{picture}(20,58)(0,-11)
   \thinlines
   \multiput(  0, 0)(20,0){2}{\line(0,1){45}}
   \multiput(  0, 0)(0,45){2}{\line(1,0){20}}
   \put(  0, 36){\makebox(10, 7){{\small $#1$}}}
   \put(  0, -9){\makebox(20, 8){{\small $#2$}}}
   \put(  0,-16){\makebox(20, 6){{\small $#3$}}}
   \thicklines
   \put(  1, 1){\oval(18,18)[tr]}
   \put(  0,10){\line(1,0){1}}
   \put( 10, 0){\circle*{2}}
   \put( 10, 0){\line(1,1){10}}
   \put( 10, 0){\circle*{2}}
\end{picture}
}
\newcommand{\trstwo}[3]{
\setlength{\unitlength}{1.6pt}
\begin{picture}(20,58)(0,-11)
   \thinlines
   \multiput(  0, 0)(20,0){2}{\line(0,1){45}}
   \multiput(  0, 0)(0,45){2}{\line(1,0){20}}
   \put(  0, 36){\makebox(10, 7){{\small $#1$}}}
   \put(  0, -9){\makebox(20, 8){{\small $#2$}}}
   \put(  0,-16){\makebox(20, 6){{\small $#3$}}}
   \thicklines
   \put(  1, 1){\oval(18,18)[tr]}
   \put(  0,10){\line(1,0){1}}
   \put( 10, 0){\circle*{2}}
   \put( 10, 0){\line(1,1){10}}
   \put( 10, 0){\circle*{2}}
   \put( 0, 18){\line(1,0){20}}
\end{picture}
}
\newcommand{\trsthr}[3]{
\setlength{\unitlength}{1.6pt}
\begin{picture}(20,58)(0,-11)
   \thinlines
   \multiput(  0, 0)(20,0){2}{\line(0,1){45}}
   \multiput(  0, 0)(0,45){2}{\line(1,0){20}}
   \put(  0, 36){\makebox(10, 7){{\small $#1$}}}
   \put(  0, -9){\makebox(20, 8){{\small $#2$}}}
   \put(  0,-16){\makebox(20, 6){{\small $#3$}}}
   \thicklines
   \put(  1, 1){\oval(18,18)[tr]}
   \put(  0,10){\line(1,0){1}}
   \put( 10, 0){\circle*{2}}
   \put( 10, 0){\line(1,1){10}}
   \put( 10, 0){\circle*{2}}
   \put( 0, 18){\line(1,0){20}}
   \put( 0, 26){\line(1,0){20}}
\end{picture}
}
\newcommand{\trsnth}[3]{
\setlength{\unitlength}{1.6pt}
\begin{picture}(20,58)(0,-11)
   \thinlines
   \multiput(  0, 0)(20,0){2}{\line(0,1){45}}
   \multiput(  0, 0)(0,45){2}{\line(1,0){20}}
   \put(  0, 36){\makebox(10, 7){{\small $#1$}}}
   \put(  0, -9){\makebox(20, 8){{\small $#2$}}}
   \put(  0,-16){\makebox(20, 6){{\small $#3$}}}
   \thicklines
   \put(  1, 1){\oval(18,18)[tr]}
   \put(  0,10){\line(1,0){1}}
   \put( 10, 0){\circle*{2}}
   \put( 10, 0){\line(1,1){10}}
   \put( 10, 0){\circle*{2}}
   \put( 0, 18){\line(1,0){20}}
   \put(10,27){\makebox(0,0){$ \vdots $}}
   \put( 0, 34){\line(1,0){20}}
\end{picture}
}
\newcommand{\sglzer}[3]{
\setlength{\unitlength}{1.6pt}
\begin{picture}(20,58)(0,-11)
   \thinlines
   \multiput(  0, 0)(20,0){2}{\line(0,1){45}}
   \multiput(  0, 0)(0,45){2}{\line(1,0){20}}
   \put(  0, 36){\makebox(10, 7){{\small $#1$}}}
   \put(  0, -9){\makebox(20, 8){{\small $#2$}}}
   \put(  0,-16){\makebox(20, 6){{\small $#3$}}}
   \thicklines
   \put( 10, 0){\line(0,1){6}}
   \put( 10, 0){\circle*{2}}
\end{picture}
}
\newcommand{\sglone}[3]{
\setlength{\unitlength}{1.6pt}
\begin{picture}(20,58)(0,-11)
   \thinlines
   \multiput(  0, 0)(20,0){2}{\line(0,1){45}}
   \multiput(  0, 0)(0,45){2}{\line(1,0){20}}
   \put(  0, 36){\makebox(10, 7){{\small $#1$}}}
   \put(  0, -9){\makebox(20, 8){{\small $#2$}}}
   \put(  0,-16){\makebox(20, 6){{\small $#3$}}}
   \thicklines
   \put( 10, 0){\line(0,1){6}}
   \put( 10, 0){\circle*{2}}
   \put(  0, 10){\line(1,0){20}}
\end{picture}
}
\newcommand{\sgltwo}[3]{
\setlength{\unitlength}{1.6pt}
\begin{picture}(20,58)(0,-11)
   \thinlines
   \multiput(  0, 0)(20,0){2}{\line(0,1){45}}
   \multiput(  0, 0)(0,45){2}{\line(1,0){20}}
   \put(  0, 36){\makebox(10, 7){{\small $#1$}}}
   \put(  0, -9){\makebox(20, 8){{\small $#2$}}}
   \put(  0,-16){\makebox(20, 6){{\small $#3$}}}
   \thicklines
   \put( 10, 0){\line(0,1){6}}
   \put( 10, 0){\circle*{2}}
   \put(  0, 10){\line(1,0){20}}
   \put(  0, 18){\line(1,0){20}}
\end{picture}
}
\newcommand{\sglnth}[3]{
\setlength{\unitlength}{1.6pt}
\begin{picture}(20,58)(0,-11)
   \thinlines
   \multiput(  0, 0)(20,0){2}{\line(0,1){45}}
   \multiput(  0, 0)(0,45){2}{\line(1,0){20}}
   \put(  0, 36){\makebox(10, 7){{\small $#1$}}}
   \put(  0, -9){\makebox(20, 8){{\small $#2$}}}
   \put(  0,-16){\makebox(20, 6){{\small $#3$}}}
   \thicklines
   \put( 10, 0){\line(0,1){6}}
   \put( 10, 0){\circle*{2}}
   \put(  0, 10){\line(1,0){20}}
   \put(10,19){\makebox(0,0){$ \vdots $}}
   \put(  0, 26){\line(1,0){20}}
\end{picture}
}
\newcommand{\dumcdots}{
\setlength{\unitlength}{1.6pt}
\begin{picture}(10,58)(0,-11)
   \put(  0, 18){\makebox(10, 6){$\cdots$}}
\end{picture}
}
\newcommand{\rsideihi}[1]{
\setlength{\unitlength}{1.6pt}
\begin{picture}(10,58)(0,-8)
   \put(0,27){{$\left. \makebox(0,10){} \right\} \scriptstyle{#1} $ }}
\end{picture}
}
\newcommand{\rsideilo}[1]{
\setlength{\unitlength}{1.6pt}
\begin{picture}(10,58)(0,-8)
   \put(0,19){{$\left. \makebox(0,10){} \right\} \scriptstyle{#1} $ }}
\end{picture}
}
\newcommand{\rsideihibig}[1]{
\setlength{\unitlength}{1.6pt}
\begin{picture}(10,58)(0,-8)
   \put(0,31){{$\left. \makebox(0,15){} \right\} \scriptstyle{#1} $ }}
\end{picture}
}